\documentclass[11pt]{article}

\usepackage{amsthm,amsmath}
\RequirePackage[numbers]{natbib}
\RequirePackage[colorlinks,citecolor=blue,urlcolor=blue]{hyperref}

\usepackage[utf8]{inputenc}
\usepackage{makeidx}
\usepackage{amsfonts}
\usepackage{amssymb}
\usepackage{gensymb}
\usepackage{caption}
\usepackage[labelformat=simple]{subcaption}
\usepackage{bbm}
\usepackage{lipsum}
\usepackage{enumitem}
\usepackage[english]{babel} 
\usepackage{bm} 
\usepackage[T1]{fontenc}	 
\usepackage{lmodern}
\usepackage{amsthm} 
\usepackage{graphicx} 
\usepackage{booktabs} 
\usepackage{multicol} 
\usepackage{tocstyle} 
\usepackage[toc,page]{appendix}

\newtheorem{theorem}{Theorem}
\newtheorem{prop}{Proposition}[section]
\theoremstyle{definition}
\newtheorem{remark}{Remark}[section]

\newcommand{\numbereqn}{\addtocounter{equation}{1}\tag{\theequation}} 
\newcommand{\mub}{\bm{\mu}}
\newcommand{\wb}{\bm{w}}
\newcommand{\Yb}{\bm{Y}}
\newcommand{\betab}{\bm{\beta}}
\newcommand{\thetab}{\bm{\theta}}
\newcommand{\xb}{\bm{x}}
\newcommand{\yb}{\bm{y}}
\newcommand{\zb}{\bm{z}}

\newcommand{\bt}{\bm{t}}
\newcommand{\vb}{\bm{v}}
\newcommand{\ab}{\bm{a}}
\newcommand{\at}{\widetilde{\bm{a}}}
\newcommand{\bbo}{\bm{b}}

\newcommand{\I}{\int_{\mathbb{R}^p} \int_{\mathbb{R}^n_-}}
\newcommand{\Wt}{\widetilde{W}}
\newcommand{\Xt}{\widetilde{X}}
\newcommand{\wt}{\widetilde{\bm{w}}}
\newcommand{\vt}{\widetilde{\bm{v}}}
\newcommand{\Ut}{\widetilde{U}}
\newcommand{\Kt}{\widetilde{K}}
\newcommand{\kt}{\widetilde{k}}
\newcommand{\sigmat}{\widetilde{\sigma}}
\newcommand{\betat}{\widetilde{\bm{\beta}}}
\newcommand{\Psit}{\widetilde{\Psi}}
\newcommand{\one}{\mathbbm{1}}
\newcommand{\diag}{\text{diag}}
\newcommand{\R}{\mathbb{R}}
\newcommand{\proj}{\mathcal{P}}
\newcommand{\TN}{\text{TN}}
\newcommand{\N}{\text{N}}
\newcommand{\var}{\text{var}}
\newcommand{\lambdamax}{\lambda_{\max}}
\newcommand{\lambdamin}{\lambda_{\min}}

\newcommand{\Z}{\mathcal{Z}}
\newcommand{\G}{\mathcal{G}}
\newcommand{\hlinehere}{\noindent\rule[0.5ex]{\linewidth}{1pt}}

\allowdisplaybreaks

\begin{document}


%
%
%
%
%
\hypersetup{linkcolor=blue}

\date{January, 2017}

\author{Saptarshi Chakraborty \\
Kshitij Khare  \thanks{Kshitij Khare  (email: kdkhare@stat.ufl.edu) is Associate Professor, Department of Statistics, University of Florida. Saptarshi Chakraborty (email: c7rishi@ufl.edu) is Graduate Student, Department of Statistics, University of Florida.} \\ \\
University of Florida}

\title{Convergence properties of Gibbs samplers for Bayesian probit regression with 
proper priors 
}

\maketitle

\begin{abstract}
The Bayesian probit regression model (Albert and Chib \cite{Albert:Chib:1993}) is popular and widely used for binary regression. While the improper flat prior for the regression coefficients is an appropriate choice in the absence of any prior information, a proper normal prior is desirable when  prior information is available or in modern high dimensional settings where the number of coefficients ($p$) is greater than the sample size ($n$). For both choices of priors, the resulting posterior density is intractable and a Data Augmentation (DA) Markov chain  is used to generate approximate samples from the posterior distribution. Establishing geometric ergodicity for this DA Markov chain is important as it provides theoretical guarantees for constructing standard errors for Markov chain based estimates of posterior quantities. In this paper, we first show that in case of proper normal priors, the DA Markov chain is geometrically ergodic \textit{for all} choices of the design matrix $X$, $n$ and $p$ (unlike the improper prior case, where $n \geq p$ and  another condition on $X$ are required for posterior propriety itself). We also derive sufficient conditions under which the DA Markov chain is trace-class, i.e., the eigenvalues of the corresponding operator are summable. In particular,  this allows us to conclude that the Haar PX-DA sandwich algorithm (obtained by inserting an inexpensive extra step in between the two steps of the DA algorithm) is strictly better than the DA algorithm in an appropriate sense.  
\end{abstract}

%
%
%
%

\section{Introduction} \label{intro}

\noindent
Let $Y_1,\cdots,Y_n$ be independent Bernoulli random variables with $P(Y_i = 1|\betab) = \Phi(\xb_i^T \betab)$ where $\xb_i \in \R^p$ is the vector of known covariates corresponding to the $i$th observation $Y_i$, for $i=1,\cdots,n$; $\betab \in \R^p$ is a vector of unknown regression coefficients and $\Phi(\cdot)$ denotes the standard normal distribution function. For $y_i \in \{0,1\};\:i=1,\cdots,n$, the likelihood is given by:
\[
P(Y_1=y_1,\cdots,Y_n=y_n)= \prod_{i=1}^n \left[\Phi(\xb_i^T\betab)\right]^{y_i} \left[1- \Phi(\xb_i^T\betab)\right]^{1- y_i}.
\]

\noindent
Our objective is to make inferences about $\betab$, and we intend to adopt a Bayesian approach as proposed in Albert and Chib \cite{Albert:Chib:1993}. In particular, we specify the prior  density $\pi(\betab)$ to be a $\N_p\left( Q^{-1}\vb, Q^{-1} \right)$ density, with a positive definite matrix $Q$ and $\vb \in \R^p$. Note that any vector $\mub \in \R^p$  can be written as $\mub = Q^{-1} Q \mub = Q^{-1}\vb$ with $\vb =  Q \mub$. Hence the assumption that the prior mean is of the form $Q^{-1}\vb$ is not restrictive. Albert and Chib \cite{Albert:Chib:1993} consider an  improper flat prior for $\betab$, which can be obtained as a limiting case of this (proper) prior, e.g., by taking $\vb = \bm 0$, and $Q$ approaching the matrix of all zeros.  Let $\yb=(y_1,\cdots,y_n)^T$ denote the observed values of the random sample $\Yb = (Y_1,\cdots,Y_n)^T$, and 
\[
m(\yb) = \int_{\R^p} \pi(\betab) \left(\prod_{i=1}^n \left[\Phi(\xb_i^T\betab)\right]^{y_i} \left[1- \Phi(\xb_i^T\betab)\right]^{1- y_i} \right) \: d\betab
\]
denote the marginal distribution of $\yb$. Then the posterior density of $\betab$ given $\Yb=\yb$ is given by
\[
\pi(\betab|\yb) = \frac{1}{m(\yb)} \pi(\betab) \left(\prod_{i=1}^n \left[\Phi(\xb_i^T\betab)\right]^{y_i} \left[1- \Phi(\xb_i^T\betab)\right]^{1- y_i} \right).
\]

\noindent
Note that the posterior density $\pi(\betab \mid {\bf y})$ does not have a closed form. It is highly intractable in the sense that computing expectations with respect to this density is not feasible in closed form, or by using numerical methods (for even moderate $p$), or by using Monte Carlo methods (for large $p$). Albert and Chib (1993) proposed a data augmentation MCMC algorithm (henceforth called the AC-DA algorithm) for this problem. As shown below, each iteration of this algorithm involves sampling from $(n+p)$ standard univariate densities. Consider the latent variables $z_1, \cdots, z_n$ where 
\[
z_i | \betab \sim N(\xb_i^T \betab, 1), \quad \text{with } y_i = \one(z_i > 0) \qquad \text{for } i =1,\cdots,n.
\] 

\noindent Further, let $X$ denote the $n \times p$ design matrix. Simple calculations show that the joint conditional density $\pi(\betab,\zb|\yb)$ of $\betab \text{ and } \zb = (z_1,\cdots,z_n)^T$ given the data $\yb$ satisfies the following 
\begin{align*} \label{pi_betaz|y}
\pi(\betab,\zb|\yb) & \propto \exp\left[ -\frac{1}{2} \left\lbrace \left(\betab - Q^{-1}\vb \right)^T Q \left(\betab - Q^{-1}\vb \right\rbrace \right) \right] \\
&\qquad \qquad \times \exp \left[ -\frac{1}{2} (\zb - X\betab)^T (\zb - X\betab) \right]\\
&\qquad \qquad \times  \prod_{i=1}^n \left \lbrace \left(\one_{(0,\infty)}(z_i)\right) ^{y_i}
												 \left(\one_{(-\infty,0]}(z_i) \right) ^{1-y_i} \right\rbrace. \numbereqn
\end{align*}

\noindent
It follows from (\ref{pi_betaz|y}) that the full conditional density of $\betab$ given $\zb, \yb$ satisfies 
\begin{align*}
\pi(\betab|\zb,\yb) &\propto \exp \left[ -\frac{1}{2} \left\lbrace \betab^T\left(X^TX + Q\right) \betab - 2 \betab^T \left(\vb + X^T\zb \right) \right\rbrace \right]. \numbereqn \label{pi_beta|others}
\end{align*}

\noindent
From $(\ref{pi_beta|others})$ we can immediately conclude that conditional on $(\zb, \yb)$, $\betab$ is normally distributed with mean vector $(X^TX + Q)^{-1}(\vb + X^T\zb)$ and covariance matrix $(X^TX + Q)^{-1}$, i.e.,
\[
\left. \betab \: \right|\zb,\yb \sim \N_p \left( \left(X^TX + Q\right)^{-1}\left(\vb + X^T\zb\right), \left(X^TX + Q\right)^{-1} \right).
\]

\noindent
Again from (\ref{pi_betaz|y}), it is easy to see that the posterior density of $\zb$ given $\betab, \yb$ satisfies
\begin{align*}
\pi(\zb|\betab,\yb)  &\propto \prod_{i=1}^n \left \lbrace \phi\left(z_i - \xb^T_i \betab \right) \:
						 \left(\one_{(0,\infty)}(z_i)\right) ^{y_i}
												 \left(\one_{(-\infty,0]}(z_i) \right) ^{1-y_i} \right\rbrace. 
\numbereqn \label{pi_z|others}
\end{align*}

\noindent
It follows that for $i=1,\cdots,n$
\[
z_i\:| \: \betab, \yb \overset{\text{indep}}{\sim} \TN \left( \xb_i^T \betab , 1 , y_i \right)
\]
where $\TN(\mu,\sigma^2,\omega)$ denotes the distribution of a truncated normal variable with mean $\mu$ and variance $\sigma^2$ which is truncated to be positive if $\omega=1$ and negative if $\omega = 0$. 

Using the standard densities above, Albert and Chib \cite{Albert:Chib:1993} construct a data augmentation  Markov chain $\Psi = (\betab_m)_{m \geq 0}$. The transition of this Markov chain from $\betab_m$ to $\betab_{m+1}$ is given as follows.\\

\noindent
\begin{minipage}{\textwidth}
\hlinehere

$(m+1)$st iteration of AC-DA Markov chain $\Psi$: 

\hlinehere

\begin{enumerate}[label = (\roman*)]
\item Draw independent $z_1, \cdots, z_n$ with 
\[
z_i \sim \TN \left( \xb_i^T \betab , 1 , y_i \right),\; i = 1,\cdots,n
\]
and call  $\zb = (z_1,\cdots,z_n)^T$.

\item Draw $\betab \sim \N_p \left( \left(X^TX + Q\right)^{-1}\left(\vb + X^T\zb \right), \left(X^TX + Q\right)^{-1} \right)$.
\end{enumerate}

\hlinehere
\end{minipage}
\\

The above conditional densities have standard forms, therefore observations from the 
Markov chain $\Psi$ can be easily generated using any standard statistical software, 
such as R (\cite{R}; \cite{MASS}).  It can be easily shown that the transition density for 
$\Psi$ is strictly positive everywhere, which implies that $\Psi$ is Harris ergodic (see 
Asmussen and Glynn \cite{asmussen:glynn:2011}). It follows that cumulative averages 
based on the above Markov chain can be used to consistently estimate corresponding 
posterior expectations. However, providing standard errors for these estimates 
requires the existence of a Markov chain CLT (which is much more challenging to 
establish than the usual CLT in the IID/independent setup). A standard method 
available to prove a Markov chain CLT involves proving that the underlying Markov 
chain is geometrically ergodic (Chan and Geyer \cite{chan:geyer:1994}; 
Flegal and Jones \cite{flegal:jones:2010}; 
Mykland, Tierney and Yu \cite{mykland:tierney:yu:1995}; Robert \cite{robert:1995}). See Section \ref{geoergodicity} for more details. 

The first contribution of this paper is a proof of geometric ergodicity for $\Psi$  for 
\emph{all choices} of the design matrix $X$, sample size $n$ and number of  
predictors $p$ . It is to be noted that the original 
Albert and Chib \cite{Albert:Chib:1993} paper has been cited over 2450 times, 
indicating the wide range of applications and studies that have been made based on 
this data augmentation (DA)  algorithm. This highlights the importance of having 
consistent standard error estimates for quantities based on the DA Markov chain. 
As we explain in Section \ref{geoergodicity}, geometric ergodicity is an important 
ingredient for obtaining a theoretical guarantee for the validity of CLT based 
standard error estimates used by practitioners.  

Data Augmentation (DA) algorithms are typically slow-mixing and take a long time to converge. However, there exist \emph{sandwich algorithms} (Meng and van Dyk \cite{Meng:vanDyk:1999}; Liu and Wu \cite{liu:wu:1999}; Hobert and Marchev \cite{hobert:marchev:2008}) which can potentially significantly improve the DA algorithm by adding just one computationally inexpensive intermediate step  ``sandwiched'' between the two steps of the DA algorithm. These sandwich algorithms are theoretically proven to be at least as good as the original DA algorithm in terms of the operator norm (see Section~\ref{trace-class}). However, to show that the sandwich algorithm is strictly better, one needs to prove some additional properties of the DA Markov chain. 

As a second major contribution of this paper, we show in Section~\ref{trace-class} that the DA Markov chain $\Psi$ is \emph{trace-class}. The derivation is quite lengthy and involved (see Section \ref{trace-class}). The fact that a DA Markov chain is trace-class  ensures  that one can construct \emph{strictly better} sandwich algorithms (e.g., Haar PX-DA algorithms; see Section~\ref{sandalg}) in the sense that the (countable) spectrum of the DA algorithm point wise dominates the (countable) spectrum of the sandwich algorithm, with at least one strict inequality (Khare and Hobert \cite{khare:hobert:2011}). We would like to point out that no results regarding trace class properties in the improper flat prior case  are available in the literature. It is to be noted that our trace class results hold both when $n \geq p$ and $n<p$,  although some sufficient conditions on  $X$ and $Q$ need to be satisfied (see Theorem~\ref{thm2}).

Roy and Hobert \cite{Roy:Hobert:2007} prove the geometric ergodicity of the resultant 
algorithm when an improper flat prior (instead of a proper normal prior) on $\betab$ is 
considered, and also derive the PX-DA sandwich algorithm in this setting. Unlike our 
paper, these authors construct minorization conditions that allow them to use 
regeneration techniques for the consistent estimation of asymptotic variances. 
On the other hand, the trace class property under the improper prior is not 
investigated in \cite{Roy:Hobert:2007}. It is important to note that an improper flat 
prior on $\betab$ leads to a proper posterior only under the following conditions 
derived in Chen and Shao \cite{chen:shao:2007}: 
\begin{enumerate}
\item $n \geq p$ and the design matrix has full column rank,
\item there exits a vector $\ab = (a_1,\cdots,a_n)^T$ with strictly positive components such that $W^T\ab = 0$, where $W$ is an $n \times p$ matrix whose $i$th row is $\xb_i^T$ or $-\xb_i^T$ according as $y_i$ is 0 or 1.
\end{enumerate}

\noindent
Roy and Hobert \cite{Roy:Hobert:2007} show that the above conditions are sufficient 
to establish geometric ergodicity as well. However, these conditions clearly exclude 
the case of modern high dimensional problems where the dimension $p$ can be 
much larger than the sample size $n$. Hence, if $p>n$ one needs to work with a 
proper normal prior. In fact, we show that when a proper normal prior is assumed, no 
assumption on $n,\: p$ and $X$ is necessary to have geometric ergodicity.
\footnote{Roy \cite{roy:2012} proves the geometric ergodicity of a DA algorithm 
based on the \emph{robit} model, which uses a Student's $t$-distribution function 
instead of the standard normal distribution function for robustness. However, this 
robustness comes at a cost of increased complexity in analysis that makes the 
problem of proving geometric ergodicity much more challenging. This is apparent from 
the  rather restrictive nature of the sufficient conditions assumed in that paper.}  
If $n \geq p$, an improper flat prior is useful in the absence of any prior information 
or for objective Bayesian inference, while the proper prior is useful in the presence of 
prior information. To the best of our knowledge, there are no general technical results 
comparing the efficiency or behavior of the AC-DA algorithm in the 
proper/improper settings when $n \geq p$. 

The remainder of this article is organized as follows. In Section~\ref{geoergodicity}, we formally define geometric ergodicity and  prove that $\Psi$ is geometrically ergodic by establishing an appropriate \emph{drift condition}. In Section~\ref{trace-class} we review the notions of trace-class Markov chains and prove that under some easily verifiable sufficient conditions $\Psi$ is trace-class. In Section~\ref{sandalg}, we briefly  review the concepts of sandwich algorithms and derive the form of one such algorithm, namely the Haar PX-DA algorithm, corresponding to the AC-DA algorithm. In Section~\ref{appl} we provide an illustration based on a real dataset to exhibit the improvements that can be achieved  by using the Haar PX-DA algorithm over the AC-DA algorithm. In Appendix~\ref{appB}, proofs of some relevant mathematical results are provided.  A method for sampling from a density that appears in the Haar PX-DA algorithm is described in Appendix~\ref{appA}.      

\section{Geometric Ergodicity for the AC-DA chain} \label{geoergodicity}
In this section we  first formally define the notion of geometric ergodicity for a Markov chain and then we show that the AC-DA chain $\Psi$ is geometrically ergodic. Let $k(\cdot,\cdot)$ denote the Markov transition density associated with $\Psi$, with corresponding Markov transition function $K(\cdot,\cdot)$. In particular, for $\betab' \in \R^p$ and a measurable set $A \in \mathcal{B}$ (:= the Borel $\sigma$-field on $\R^p$), $K(\betab',A) = \int_{A} k(\betab',\betab)\:d\betab$. For $m \geq 1$, the corresponding $m-$step Markov transition function is defined in the following inductive fashion.
\[
K^m(\betab',A) = \int_{\R^p} K^{m-1}(\betab,A)\:K(\betab',d\betab) = \text{Pr}(\betab_{m+j} \in A|\betab_j = \betab')
\]
for all $j=0,1,2,\dots$; with $K^1 \equiv K$. Let $\Pi(\cdot|\yb)$ denote the probability measure associated with the posterior density $\pi(\betab|\yb)$, so that $\Pi(A|\yb) = \int_{A} \pi(\betab|\yb) d\betab$. Here $\pi(\betab|\yb)$ denotes the $\betab-$marginal of the joint density $\pi(\betab,\zb|\yb)$.  The chain $\Psi$ is geometrically ergodic if there exist a constant $\eta \in [0,1)$ and a function $Q:\R^p \rightarrow [0,\infty)$ such that for any $\betab \in \R^p$ and any $m \in \mathbb{N}$,
\[
\|K^m(\betab,\cdot) - \Pi(\cdot|\yb) \| := \sup_{A \in \mathcal{B}}|K^m(\betab,A) - \Pi(A|\yb) | \leq Q(\betab) \eta^m
\]
As mentioned in the introduction, geometric ergodicity implies existence of a CLT for 
Markov chain based cumulative averages. In particular, let $g \in L^2(\pi(\betab|\yb))$ 
such that $E_\pi g(\betab)^{2} < \infty$, and let $(\betab_N)_{N=1}^m$ denote the 
observations generated by the DA algorithm. Define $\bar{g}_m := m^{-1} \sum_{N=1}
^m g\left(\betab_N \right)$. If the (reversible) DA Markov chain is geometrically 
ergodic, then there exists $\sigma^2_{g} \in (0, \infty)$ such that $\sqrt{m}(\bar{g}_m - 
E_\pi g) \xrightarrow{d} N(0, \sigma^2_{g})$ as $m \rightarrow \infty$. Several 
methods for obtaining consistent estimators of $\sigma^2_g$ are available in the 
literature, see for example \cite{JHCN:2006,flegal:jones:2010}. These methods 
typically require additional moment assumptions on $g$ along with other mild 
regularity assumptions. 

The following theorem establishes geometric ergodicity of $\Psi$ by forming a (geometric) drift condition on the basis of the following (drift) function
\begin{align*}
\nu(\betab) = \betab^T \left( X^TX + Q \right) \betab
\end{align*}
The fact that $( X^TX + Q )$ is positive definite ensures that $\nu(\betab)$ is unbounded off compact sets as a function of $\betab$, i.e., for each $\alpha > 0$, the level set $\{\betab : \betab^T(X^TX + Q)\betab \leq \alpha\}$ is compact.

\begin{theorem}\label{thm1}
Let $k(\cdot, \cdot)$ denote the transition density corresponding to the Markov chain $\Psi$. Then for any arbitrary $\betab' \in \R^p$ representing the current state, there exists $\rho \in (0,1)$ and $L \in \R$ such that
\begin{align} \label{driftcond}
\int_{\R^p} \nu(\betab) k(\betab', \betab) \:d \betab \leq \rho \nu(\betab') + L.
\end{align} 
\end{theorem}

\begin{proof}

On the outset, note that the transition density corresponding to the AC-DA Markov chain is given by
\[
k(\betab', \betab) = \int_{\Z} \pi(\betab | \zb,\yb) \pi(\zb | \betab',\yb) \:d\zb, 
\]

\noindent
where $\Z$ denotes the space where the random vector $\zb$ lives, i.e., $\Z$ is the Cartesian product of $n$ half lines $\R_+ = (0, \infty)$ or $\R_- = (-\infty, 0]$ according as $y_i = 1$ or $0$. Therefore, by Fubini's theorem, we get
\begin{align*}\label{E(betagivenall)}
\int_{\R^p} \nu(\betab) k(\betab', \betab) \:d \betab &= \int_{\Z} \left \lbrace { \int_{\R^p} \nu(\betab) \: \pi(\betab|\zb,\yb) \: d\betab } \right \rbrace \pi(\zb | \betab',\yb)\:d\zb. \numbereqn
\end{align*}

\noindent
The inner integral in the right hand side of (\ref{E(betagivenall)}) is given by
\begin{align*}\label{E(beta)}
\quad \int_{\R^p} \nu(\betab) \: \pi(\betab|\zb,\yb) \: d\betab
	&= E \left( \left. \betab^T \left( X^TX +Q \right) \betab \: \right| \zb,\yb \right) \\
	&= \text{trace} \left( \left( X^TX +Q \right) \var(\betab|\zb,\yb) \right) \\ 
	&\qquad +  E(\betab| \zb,\yb)^T \left( X^TX +Q \right) E(\betab| \zb,\yb)  \\
	&= p + \left( X^T\zb +\vb \right)^T \left( X^TX +Q \right)^{-1} \left( X^T\zb +\vb \right)  \\
	&= p + \left\| \left( X^TX +Q \right)^{-1/2} \left( X^T\zb +\vb \right) \right\|^2.  \numbereqn \end{align*} 

\noindent
Note that for any  $\ab \in \R^p, \bbo \in \R^p$ and $c>0$,
\begin{align*}
\| \ab + \bbo \|^2
	&= \|\ab\|^2 + \|\bbo\|^2 + 2 (c \ab)^T \left(\frac{1}{c} \: \bbo\right) \\
	&\leq \|\ab\|^2 + \|\bbo\|^2 + c^2 \|\ab\|^2 + \frac{1}{c^2} \|\bbo\|^2 \\
	&= \left(1+c^2\right) \|\ab\|^2 + \left( 1+ \frac{1}{c^2}\right) \|\bbo\|^2.
\end{align*}  

\noindent Therefore, by taking $\ab =  \left( X^TX +Q \right)^{-1/2}  X^T\zb$, $\bbo = \left( X^TX +Q \right)^{-1/2} \vb$ and any $c > 0$, we get the following upper bound for $(\ref{E(beta)})$: 
\begin{align*}\label{E(beta)upperbd}
&\quad  p + \left(1+c^2\right) \left\| \left( X^TX +Q \right)^{-1/2}  X^T\zb \: \right\|^2 + \left( 1+ \frac{1}{c^2}\right) \left\| \left( X^TX +Q \right)^{-1/2} \vb \: \right\|^2 \\
&= \left(1+c^2\right) \left\| \left( X^TX +Q \right)^{-1/2}  X^T\zb \: \right\|^2 + A_1 \numbereqn
\end{align*}
where
\[
A_1 = A_1(c) = p + \left( 1+ \frac{1}{c^2}\right) \left\| \left( X^TX +Q \right)^{-1/2} \vb \: \right\|^2.
\]

\noindent Hence, from (\ref{E(betagivenall)}), (\ref{E(beta)}) and $(\ref{E(beta)upperbd})$ we can write for any $c>0$,
\begin{align*}\label{fullupperbd}
& \quad \int_{\Z} \left \lbrace { \int_{\R^p} \nu(\betab) \: \pi(\betab|\zb,\yb) \: d\betab } \right \rbrace \pi(\zb | \betab',\yb)\:d\zb \\
& \leq \left(1+c^2\right) \int_{\R^n} \left\| \left( X^TX +Q \right)^{-1/2}  X^T\zb \: \right\|^2 \: \pi(\zb|\betab,\yb)\: d\zb \: +  \:A_1 \\
&= \left(1+c^2\right) E\left( \left. \zb^TX \left( X^TX +Q \right)^{-1}  X^T\zb \: \right| \betab',\yb \right) \: + \: A_1. \numbereqn
\end{align*}

\noindent Now, note that
\begin{align*}\label{E(ztXothers)}
E\left( \left. \zb^TX \left( X^TX +Q \right)^{-1}  X^T\zb \: \right| \betab',\yb \right) 
	&= E\left( \left. \zb^T\Xt \left( \Xt^T\Xt + I_p \right)^{-1}  \Xt^T\zb \: \right| \betab',\yb \right) \\
	&\leq \lambdamax \: E \left( \left. \zb^T \zb \right| \betab', \yb \right)\\
	&= \lambdamax \: \sum_{i=1}^n E \left( \left. z_i^2 \: \right| \betab', \yb \right). \numbereqn
\end{align*}
Here $\Xt = X Q^{-1/2}$, and $\lambdamax = \lambdamax(\Xt\: ( \Xt^T\Xt + I_p)^{-1} \: \Xt^T)$ denotes the largest eigenvalue of $\Xt\: ( \Xt^T\Xt + I_p)^{-1} \: \Xt^T$. Standard results from the theory of truncated normal distributions show that (see Roy and Hobert \citep{Roy:Hobert:2007})
\[
U \sim \TN(\xi,1,1) \implies E U^2 = 1 + \xi^2 + \frac{\xi \phi(\xi)}{\Phi(\xi)}
\]
and
\[
U \sim \TN(\xi,1,0) \implies E U^2 = 1 + \xi^2 - \frac{\xi \phi(\xi)}{1-\Phi(\xi)}.
\]

\noindent Therefore, it follows that for all $i = 1,\cdots,n$
\[
E \left( \left. z_i^2 \: \right| \betab', \yb \right) =
\begin{cases}
1 + \left(\xb_i^T\betab'\right)^2 + \frac{\left(\xb_i^T\betab'\right) \phi\left(\xb_i^T\betab'\right)}{\Phi\left(\xb_i^T\betab'\right)} & \text{if } y_i = 1 \\
1 + \left(\xb_i^T\betab'\right)^2 - \frac{\left(\xb_i^T\betab'\right) \phi\left(\xb_i^T
\betab'\right)}{1-\Phi\left(\xb_i^T\betab'\right)} & \text{if } y_i = 0
\end{cases}.
\]

\noindent A more compact way of expressing this is as follows.
\[
E \left( \left. z_i^2 \: \right| \betab', \yb \right) = 1 + \left(\wb_i^T\betab'\right)^2 - \frac{\left(\wb_i^T\betab'\right) \phi\left(\wb_i^T\betab'\right)}{1-\Phi\left(\wb_i^T\betab'\right)}
\]
where
\[
\wb_i = 
\begin{cases}
\xb_i & \text{if } y_i = 0 \\
-\xb_i & \text{if } y_i = 1
\end{cases}.
\]

\noindent Now, for all $i=1,\cdots,n$, 
\begin{align*}
-\frac{\left(\wb_i^T\betab'\right) \phi\left(\wb_i^T\betab'\right)}{1-\Phi\left(\wb_i^T\betab'\right)}
&\leq 
\begin{cases}
\left| \frac{\left(\wb_i^T\betab'\right) \phi\left(\wb_i^T\betab'\right)}{1-\Phi\left(\wb_i^T\betab'\right)} \right| & \text{ if }\wb_i^T\betab' \leq 0 \\
0 & \text{ if } \wb_i^T\betab' > 0 
\end{cases}\\
&\leq \sup_{u \in (-\infty,0]} \left| \frac{u \phi\left(u \right)}{1-\Phi\left(u \right)} \right|
=: \Lambda
\end{align*}
and it is clear that $\Lambda \in (0,\infty)$. This implies, for all $i = 1, \cdots, n$
\[
E \left( \left. z_i^2 \: \right| \betab', \yb \right) \leq 1 + \left(\wb_i^T\betab'\right)^2 + \Lambda
= 1 + \left(\xb_i^T\betab'\right)^2 + \Lambda.
\]

\noindent Therefore, from $(\ref{E(ztXothers)})$ we can write,
\begin{align*} \label{geofinal}
E\left( \left. \zb^TX \left( X^TX +Q \right)^{-1}  X^T\zb \: \right| \betab',\yb \right) 
&< \lambdamax \: \sum_{i=1}^n \left(\xb_i^T\betab'\right)^2 + A_2 \\
&= \lambdamax \: {\betab'}^T \left( X^TX \right) \betab' + A_2 \\
&\leq \lambdamax \: {\betab'}^T \left( X^TX +Q \right) \betab' + A_2 \\
&= \lambdamax \: \nu(\betab') + A_2 \numbereqn
\end{align*} 
where $A_2 = n \: \lambdamax \: (1+\Lambda) < \infty$, and inequality in the  second last line follows from the fact that $Q$ is positive definite.\\

\noindent Finally, combining  (\ref{fullupperbd}), and (\ref{geofinal}), we get
\begin{align*}
\int_{\R^p} \nu(\betab) k(\betab', \betab) \:d \betab 
&\leq (1+c^2)\: \lambdamax\: \nu(\betab') + L \\
&= \rho(c)\: \nu(\betab') + L  
\end{align*}
where $L = L(c) = A_1(c) + (1+c^2)A_2$, $\rho(c) = (1+c^2)\lambdamax$ and $c>0$ is arbitrary. It remains to show that there exists $c > 0$ such that $0< \rho(c) <1$. It follows from Proposition~\ref{prop1} in Appendix~\ref{appB}  that $\lambdamax \in (0,1)$. So, for any $c$, $\rho(c) = \lambdamax(1+c^2) >0$. To show that there exits $c>0$ such that $\rho(c) < 1$, take any $c_0 \in \left(0, \sqrt{\lambdamax^{-1}-1}\right)$ $\left(\text{e.g. } c_0 = \frac{1}{2}\sqrt{\lambdamax^{-1} - 1} \right)$.  Then
\[
1+c_0^2 < \lambdamax^{-1} \implies \rho(c_0) = \lambdamax (1+c_0^2) < 1.
\]

\noindent This completes the proof.

\end{proof}

\begin{remark}
As mentioned earlier, since $(X^TX + Q)$ is always positive definite, $\nu(\betab)$ is unbounded off compact sets for any design matrix $X$. Therefore, from Meyn and Tweedie  \cite[Lemma 15.2.8]{meyn:tweedie:1993} and Theorem~\ref{thm1}, it follows that for \textit{any} $X$, $n$ and $p$, the AC-DA Markov chain $\Psi$ is geometrically ergodic. 
\end{remark}

\section{Trace-class property for the AC-DA chain} \label{trace-class}

Recall that the AC-DA Markov chain $\Psi$ has associated transition density given by 
\begin{align}\label{mtd}
k(\betab', \betab) = \int_{\mathcal{Z}} \pi(\betab | \zb,\yb) \pi(\zb | \betab',\yb) \:d\zb. 
\end{align}
Let $L_0^2 (\pi(\cdot \mid {\bf y}))$ denote the space of square-integrable functions with mean zero (with respect to  the posterior density 
$\pi(\betab \mid {\bf y})$). Let $K$ denote the Markov operator on $L_0^2 (\pi(\cdot \mid {\bf y}))$ associated with the transition 
density $k$. Note that the Markov transition density $k$ is reversible with respect to its invariant distribution, and $K$ is a positive, 
self-adjoint operator. The operator $K$ is {\it trace class} (see J{\"o}rgens \cite{jorgens:1982}) if 
\begin{align} \label{tracegen}
\int_{\R^p} k(\betab, \betab) \; d \betab < \infty.
\end{align} 
If the trace-class property holds, then $K$ is compact, and its eigenvalues are summable (stronger than square 
summable), which in particular also implies that the associated Markov chain is geometrically ergodic. The trace class property for 
a DA Markov chain has another important implication. Hobert and Marchev \cite{hobert:marchev:2008} define a class of sandwich 
algorithms called Haar PX-DA algorithms which they show to be optimal in an appropriate sense. If a DA algorithm is trace class, 
then so is the Haar PX-DA algorithm. Furthermore, the spectrum of the Haar PX-DA operator is \emph{strictly better} than the DA algorithm in the 
sense that the (countable) spectrum for the Haar PX-DA algorithm is dominated pointwise by the spectrum of the DA algorithm, with at least one strict domination (Khare and Hobert \cite{khare:hobert:2011}). See Section \ref{sandalg} for more details. 

The following theorem provides sufficient conditions under which the Markov operator $K$ corresponding to the AC-DA algorithm 
is trace class. 
\begin{theorem} \label{thm2}
Let $X$, the design matrix, have either full column rank (if $n \geq p$) or full row rank (if $n < p$). Then, the AC-DA Markov chain 
$\Psi$ is trace-class if 
\begin{enumerate}[label = (\Alph*)]
\item \label{trcond1} All eigenvalues (or all non-zero eigenvalues, if $n < p$) of $Q^{-1/2}X^TX Q^{-1/2}$ are less than $7/2$, OR 
\item \label{trcond2} $X Q^{-1/2}$ is rectangular diagonal. 
\end{enumerate}
\end{theorem}

\begin{proof}

\noindent
We shall show that (\ref{tracegen}) holds for the Markov chain $\Psi$ if either  
\ref{trcond1} or \ref{trcond2} holds. Conditions \ref{trcond1} and 
\ref{trcond2} will not play a role at all in the first half of this proof, but will be 
needed to show the positivity of an appropriate function in the second half of 
the proof. 

\noindent First, note that (\ref{pi_z|others}) implies
\begin{align*}
\pi(\zb|\betab,\yb) &\propto \exp\left(-\frac{1}{2} \sum_{i=1}^n (z_i - \xb_i^T \betab)^2 
\right) \\
& \qquad \times \prod_{i=1}^n \left\lbrace \left(\frac{1}{\Phi(\xb_i^T\betab)}\right)^{y_i} 
\left(\frac{1}{1- \Phi(\xb_i^T\betab)}\right)^{1 - y_i} \right\rbrace\\
&=	\exp\left(-\frac{1}{2}  (\zb - X \betab)^T (\zb - X \betab) \right) \\ 
& \qquad \times \prod_{i=1}^n \left\lbrace \left(\frac{1}{\Phi(\xb_i^T\betab)}\right)^{y_i} 
\left(\frac{1}{1- \Phi(\xb_i^T\betab)}\right)^{1 - y_i} \right\rbrace. 
\end{align*}

\noindent Now, let us define for $i = 1, \cdots, n$,
\[
t_i =
\begin{cases}
z_i , & \text{if } y_i = 0 \\
-z_i , & \text{if } y_i = 1
\end{cases}
;\quad
\wb_i = 
\begin{cases}
\xb_i , & \text{if } y_i = 0 \\
-\xb_i , & \text{if } y_i = 1
\end{cases}
;\quad \text{and }
W_{n \times p} =
\begin{pmatrix}
\wb_1^T \\
\vdots \\
\wb_n^T
\end{pmatrix}.
\]

\noindent Then $X^TX = W^TW$ and $X^T\zb = W^T\bt$, and absolute value of the Jacobian of the transformation $\zb \rightarrow \bt$ is one. So the conditional density $\pi(\bt|\betab,\yb)$ of $\bt$ given $\betab, \yb$ satisfies 
\begin{align*} \label{pi_t}
\pi(\bt|\betab,\yb) & \propto \exp\left(-\frac{1}{2} \{ \bt^T\bt - 2 \betab^T W^T \bt + \betab^T W^T W \betab  \}  \right) \\
& \quad \times  \prod_{i=1}^n \left(\frac{1}{1- \Phi(\wb_i^T\betab)}\right). \numbereqn
\end{align*}

\noindent Again, simple calculations on (\ref{pi_betaz|y}) show that
\begin{align*}
\pi(\betab | \zb,\yb) 
&\propto \exp \left[ -\frac{1}{2} \left\lbrace \betab^T(X^TX+Q)\betab - 2\betab^T(X^T\zb + \vb) \right. \right. \\ 
& \qquad \qquad \quad + \left. \left. (X^T\zb+\vb)^T (X^TX+Q)^{-1} (X^T\zb+\vb) \right\rbrace \right] \\
& \propto \exp \left[ -\frac{1}{2} \left\lbrace \betab^T(X^TX+Q)\betab - 2\betab^T X^T \zb \right. \right. \\
& \qquad \qquad \quad - \left. \left. 2 \betab^T \vb +   \zb^T X (X^TX+Q)^{-1} X^T \zb \right. \right. \\
& \qquad \qquad \quad + \left. \left. 2 \zb^T X (X^TX+Q)^{-1} \vb\right\rbrace \right].
\end{align*}
so that
\begin{align} \label{pi_beta}
\pi(\betab | \bt,\yb) 
& \propto \exp \left[ -\frac{1}{2} \left\lbrace \betab^T(W^TW+Q)\betab - 2\betab^T W^T \bt  \right. \right. \nonumber \\
& \qquad \qquad - \left. \left.  2 \betab^T \vb + \bt^T W (W^T W + Q)^{-1} W^T \bt  \right. \right. \nonumber \\ 
& \qquad \qquad - \left. \left.  2 \bt^T W (W^T W + Q)^{-1} \vb \right\rbrace \right].
\end{align}

\noindent  Therefore, using (\ref{mtd}), (\ref{pi_t}) and (\ref{pi_beta}), we get the following form for the integral in (\ref{tracegen}) in the current setting.
\begin{align} \label{I_expr1}
I &:= \int_{\mathbb{R}^p} \int_{\mathcal{Z}} \pi(\betab|\zb,\yb) \pi(\zb|\betab,\yb) \:d\zb \:d\betab \nonumber \\
&= \I \pi(\betab|\bt,\yb) \pi(\bt|\betab,\yb) \:d\bt \:d\betab \nonumber \\
&= C_0 \I   \exp \left[ -\frac{1}{2} \left\lbrace \betab^T(W^TW+Q)\betab - 2\betab^T W^T \bt  \right. \right. \nonumber \\
& \qquad \qquad \qquad \qquad \qquad \quad - \left. \left.  2 \betab^T \vb + \; \bt^T W (W^T W + Q)^{-1} W^T \bt \right. \right. \nonumber \\ 
& \qquad \qquad \qquad \qquad \qquad \quad + \left. \left. 2 \bt^T W (W^T W + Q)^{-1} \vb \right\rbrace \right] \nonumber \\
& \qquad \qquad \qquad \quad \times \exp\left(-\frac{1}{2} \{ \bt^T\bt - 2 \betab^T W^T \bt + \betab^T W^T W \betab  \}  \right)  \nonumber \\
& \qquad \qquad \qquad \quad \times \prod_{i=1}^n \left(\frac{1}{1- \Phi(\wb_i^T\betab)}\right) \:d\bt \:d\betab.
\end{align}

\noindent Here $C_0$ denotes the product of all constant terms (independent of $\betab$ and $\bt$) appearing in the full conditional densities $\pi(\betab|\bt,\yb)$ and $\pi(\bt|\betab,\yb)$. Let us define $\thetab = Q^{-1/2} \betab$, $\Wt = W Q^{-1/2}$ and $\vt = Q^{-1/2} \vb$. Absolute value of the Jacobian of the transformation $\betab \rightarrow \thetab$ is $\{\det(Q)\}^{-1/2} > 0$. Therefore, the right hand side of $(\ref{I_expr1})$ is proportional to 
\begin{align} \label{I_expr2}
& \I  \exp \left[ -\frac{1}{2} \left\lbrace \thetab^T(\Wt^T \Wt + I_p)\thetab - 2\thetab^T \Wt^T \bt - 2 \thetab^T \vt \right. \right. \nonumber \\
& \qquad \qquad \qquad \quad \left. \left. +\; \bt^T \Wt (\Wt^T W + I_p)^{-1} \Wt^T \bt + 2 \bt^T \Wt (\Wt^T \Wt + I_p)^{-1} \vt \right\rbrace \right] \nonumber \\
&  \qquad \qquad \times \exp\left(-\frac{1}{2} \{ \bt^T\bt - 2 \thetab^T \Wt^T \bt + \thetab^T \Wt^T \Wt \thetab  \}  \right)  \nonumber \\
&  \qquad \qquad \times \prod_{i=1}^n \left(\frac{1}{1- \Phi(\wt_i^T\thetab)}\right) \:d\bt \:d\thetab \nonumber \\
&= \int_{\mathbb{R}^p} \frac{\exp \left[ -\frac{1}{2} \left\lbrace \thetab^T(2\Wt^T \Wt + I_p)\thetab - 2 \thetab^T \vt \right\rbrace \right]} {\prod_{i=1}^n \left(1- \Phi(\wt_i^T\thetab)\right)} \nonumber \\
& \qquad  \quad \left( \int_{\mathbb{R}^n_-} \exp\left[ 2 \thetab^T \Wt^T \bt  - \bt^T \Wt (\Wt^T \Wt + I_p)^{-1} \vt \right. \right. \nonumber \\
& \qquad  \qquad \qquad \qquad -\frac{1}{2} \left. \left.  \bt^T \left( I_n + \Wt (\Wt^T \Wt + I_p)^{-1} \Wt^T\right) \bt \right] \:d\bt  \right) \:d\betab.
\end{align} 

\noindent Now consider the partition 
\[
\mathbb{R}^p = \biguplus_{\zeta \subseteq \{1,\cdots,n\} } A_\zeta
\]
where
\begin{align*} 
A_\zeta = \{\thetab : \wt_i^T \thetab > 0 \text{ if } i \in \zeta \text{ and } \wt_i^T\thetab \leq 0 \text{ if } i  \notin \zeta  \}.
\end{align*}

\noindent 
The above partition is essentially obtained by using the $n$ hyperplanes defined by $\wt_i^T \thetab = 0$ for $1 \leq i \leq n$. This partition has also been used in 
\cite{Roy:Hobert:2007} for proving geometric ergodicity of the DA Markov chain corresponding to an improper flat prior on $\betab$. The right hand side of $(\ref{I_expr2})$ 
can now be written as 
\begin{align*}
& \sum_{\zeta \subseteq \{1,\cdots,n\}} \int_{A_\zeta} \frac{\exp \left[ -\frac{1}{2} \left\lbrace \thetab^T(2\Wt^T \Wt + I_p)\thetab - 2 \thetab^T \vt \right\rbrace \right]}{\prod_{i=1}^n \left(1- \Phi(\wt_i^T\thetab)\right)} \\
& \qquad \qquad \quad \left( \int_{\mathbb{R}^n_-} \exp\left[ 2 \thetab^T \Wt^T \bt  - \bt^T \Wt (\Wt^T W + I_p)^{-1} \vt \right. \right. \\
& \qquad \qquad \qquad \qquad \qquad -\frac{1}{2} \left. \left.  \bt^T \left( I_n + \Wt (\Wt^T \Wt + I_p)^{-1} \Wt^T\right) \bt \right] \:d\bt  \right) \:d\thetab  \\
&= \sum_{\zeta \subseteq \{1,\cdots,n\}} I_{A_\zeta}, \text{ say} 
\end{align*}
where
\begin{align*}\label{I_Azeta}
I_{A_\zeta} &= \int_{A_\zeta} \frac{\exp \left[ -\frac{1}{2} \left\lbrace \thetab^T(2\Wt^T \Wt + I_p)\thetab - 2 \thetab^T \vt \right\rbrace \right]}{\prod_{i=1}^n \left(1- \Phi(\wt_i^T\thetab)\right)} \\
&  \qquad \quad \left( \int_{\mathbb{R}^n_-} \exp\left[ 2 \thetab^T \Wt^T \bt  - \bt^T \Wt (\Wt^T W + I_p)^{-1} \vt \right. \right. \\
&  \qquad \qquad \qquad \qquad -\frac{1}{2} \left. \left.  \bt^T \left( I_n + \Wt (\Wt^T \Wt + I_p)^{-1} \Wt^T\right) \bt \right] \:d\bt  \right) \:d\thetab.  \numbereqn
\end{align*}

\noindent Therefore to prove (\ref{tracegen}), it is enough to show that for any ${\zeta \subseteq \{1,\cdots,n\}}$  
\begin{align}\label{I_a_cond}
I_{A_\zeta} < \infty.
\end{align}

\noindent Fix an arbitrary  $\zeta \subseteq \{1,\cdots,n\}$. Define  $a_{1i} = \wt_i^T\thetab \; \one(\wt_i^T\thetab \leq 0) = \wt_i^T\thetab \; \one(i \notin \zeta)$ and $a_{2i} = \wt_i^T\thetab \; \one(\wt_i^T\thetab > 0) = \wt_i^T\thetab \; \one(i \in \zeta)$, for $i = 1,\cdots,n$ and $\ab_j = (a_{ji})_{1 \leq i \leq n}$ for $j=1,2$. This means $\ab_1 + \ab_2 = \Wt \thetab$ and $a_{1i}a_{2i} = 0, \text{ for all } i$. In particular, $\ab_1^T\ab_2 = 0$. Then for $i \notin \zeta$
\[
\wt_i^T\thetab \leq 0 \implies \Phi(\wt_i^T\thetab) \leq \frac{1}{2} \implies \frac{1}{1-\Phi(\wt_i^T\thetab)} \leq 2
\]
and for $i \in \zeta$
\begin{align*}
 &\wt_i^T\thetab > 0\implies  \frac{2}{ (\wt_i^T\thetab) + \sqrt{4+(\wt_i^T\thetab)^2} } < \frac{1-\Phi(\wt_i^T\thetab)}{\phi(\wt_i^T\thetab)} \quad  \text{(Birnbaum \cite{birnbaum:1942})} \\
\implies  &\frac{1}{1-\Phi(\wt_i^T\thetab)} < \frac{(\wt_i^T\thetab) + \sqrt{4+(\wt_i^T\thetab)^2} }{2 \: \phi(\wt_i^T\thetab)} = \frac{q(\wt_i^T\thetab) }{ \phi(\wt_i^T\thetab)}
\end{align*}
where $q(x) = (x + \sqrt{4 +x^2})/2 $, and $\phi(\cdot)$ denote the standard normal density function. Thus, for any $i = 1, \cdots, n$,
\begin{align*}
\frac{1}{1-\Phi(\wt_i^T\thetab)} &< 2^{\one(i \notin \zeta)} \left\lbrace \frac{q(\wt_i^T\thetab) }{ \phi(\wt_i^T\thetab)} \right\rbrace ^ {\one(i \in \zeta)} \\
&\leq 2 \left\lbrace \frac{q(\wt_i^T\thetab) }{ \phi(\wt_i^T\thetab)} \right\rbrace ^ {\one(i \in \zeta)} \\
&= 2 \{q(\wt_i^T\thetab) \sqrt{2 \pi}\} ^ {\one(i \in \zeta)} \exp \left[ \frac{1}{2} (\wt_i^T\thetab)^2 \one(i \in \zeta) \right] \\
&\leq 2\left(1 + \sqrt{2\pi}\: q\left(\wt_i^T\thetab \; \one(i \in \zeta)\right) \right) \exp \left[ \frac{1}{2} (\wt_i^T\thetab)^2 \one(i \in \zeta) \right] \\
&= 2 \left(1 + \sqrt{2\pi}\: q(a_{2i})\right) \exp \left[ \frac{1}{2} a_{2i} ^2 \right] \\
&= \tilde{q}(a_{2i}) \exp \left[ \frac{1}{2} a_{2i} ^2 \right],\text{ say}
\end{align*}
where $\tilde{q}(x) = 2 \left(1 + \sqrt{2\pi} \: q(x)\right)$. Therefore,
\begin{align*} \label{milsubd}
\prod_{i=1}^n \left(\frac{1}{1- \Phi(\wt_i^T\thetab)}\right) & < \left(\prod_{i=1}^n \tilde{q}
(a_{2i}) \right) \exp \left[ \frac{1}{2} \ab_2^T \ab_2 \right] \\ 
&= \widetilde{Q}(\ab_2) \exp 
\left[ \frac{1}{2} \ab_2^T \ab_2 \right], \text{ say}, \numbereqn
\end{align*}
where $\widetilde{Q}(\ab_2) = \prod_{i=1}^n \tilde{q}(a_{2i})$.\\

\noindent We now derive an upper bound for the inner integral in (\ref{I_Azeta}). Let  $\epsilon \in (0,1)$ be arbitrary and 
\[
\vb^* = \left( I_n + \Wt (\Wt^T \Wt + I_p)^{-1} \Wt^T\right)^{-1/2}\: \Wt (\Wt^T \Wt + I_p)^{-1} \vt.
\] 
Then using the fact that $ 2 |\ab^T\bbo| \leq \ab^T\ab + \bbo^T\bbo$ with 
\[
\ab = \sqrt{\epsilon} \left( I_n + \Wt (\Wt^T \Wt + I_p)^{-1} \Wt^T\right)^{1/2} \bt 
\text{ and } 
\bbo = \vb^*/\sqrt{\epsilon}
\]
we get
\begin{align*}
& \int_{\R^n_-} \exp\left[ 2 \thetab^T \Wt^T \bt  - \bt^T \Wt (\Wt^T W + I_p)^{-1} \vt \right.   \\
& \qquad \qquad \qquad - \left. \frac{1}{2}  \: \bt^T \left( I_n + \Wt (\Wt^T \Wt + I_p)^{-1} \Wt^T\right) \bt \right] \:d\bt   \\
&\leq \int_{\R^n_-} \exp \left[ 2 \thetab^T \Wt^T \bt  +\frac{\epsilon}{2} \: \bt^T \left( I_n + \Wt (\Wt^T \Wt + I_p)^{-1} \Wt^T\right) \bt + \frac{1}{2\epsilon} \: {\vb^*}^T \vb^* \right. \\
& \qquad \qquad \qquad - \left.  \frac{1}{2} \: \bt^T \left( I_n + \Wt (\Wt^T \Wt + I_p)^{-1} \Wt^T\right) \bt \right] \:d\bt\\
&= C_1 \int_{\R^n_-} \exp \left[ 2 \thetab^T \Wt^T \bt - \frac{1}{2}(1-\epsilon) \: \bt^T \left( I_n + \Wt (\Wt^T \Wt + I_p)^{-1} \Wt^T\right) \bt \right] \:d\bt\\ 
&= C_1 \int_{\R^n_-} \exp \left[\: 2 \left( \sum_{i \in \zeta} t_i\:\wt_i^T \thetab + \sum_{i \notin \zeta} t_i\:\wt_i^T \thetab \right) \right.\\
& \qquad \qquad \qquad \qquad - \left.  \frac{1}{2} (1-\epsilon) \: \bt^T \left( I_n + \Wt (\Wt^T \Wt + I_p)^{-1} \Wt^T\right) \bt \right] \:d\bt \\ 
&\leq C_1 \int_{\R^n_-} \exp \left[\: 2  \sum_{i \notin \zeta} t_i\:\wt_i^T \thetab - \frac{1}{2} (1-\epsilon) \: \bt^T \left( I_n + \Wt (\Wt^T \Wt + I_p)^{-1} \Wt^T\right) \bt \right] \:d\bt \\
& \qquad \qquad \qquad \qquad \qquad \qquad \left(\text{since }\sum_{i \in \zeta} t_i\:\wt_i^T \thetab \leq 0 \text{ for } \bt \in \R^n_-\right) \\ 
&= C_1 \int_{\R^n_-} \exp \left[\: 2 \ab_1^T \bt - \frac{1}{2} (1-\epsilon) \: \bt^T \left( I_n + \Wt (\Wt^T \Wt + I_p)^{-1} \Wt^T\right) \bt \right] \:d\bt\\ 
&= C_1 \times  \exp \left[ \frac{1}{2} \left(\frac{4}{1-\epsilon}\right) \ab_1^T  \left( I_n + \Wt (\Wt^T \Wt + I_p)^{-1} \Wt^T\right)^{-1} \ab_1 \right] \times (C_1')^{-1} \\
& \qquad \times \int_{\R^n_-} C_1' \exp \left[ - \frac{1}{2} (1-\epsilon) \right. \\
& \qquad \qquad \qquad \qquad \qquad \times \left. (\bt - \ab_1^*)^T \left( I_n + \Wt (\Wt^T \Wt + I_p)^{-1} \Wt^T\right) (\bt - \ab_1^*) \right] \:d\bt\\
& \qquad \qquad \qquad \qquad \left(\text{where } \ab_1^* = \left(\frac{2}{1-\epsilon}\right)  \left( I_n + \Wt (\Wt^T \Wt + I_p)^{-1} \Wt^T\right)^{-1} \ab_1\right)\\
&\leq C_1'' \: \exp \left[ \frac{1}{2} \left(\frac{4}{1-\epsilon}\right) \ab_1^T  \left( I_n + \Wt (\Wt^T \Wt + I_p)^{-1} \Wt^T\right)^{-1} \ab_1 \right] \numbereqn \label{innrint}
\end{align*}
where \begin{align*}
C_1 &= \exp\left(\frac{{\vb^*}^T \vb^*}{2\epsilon}\right) \\
C_1' &= (2\pi)^{-n/2} (1-\epsilon)^{n/2} \left\lbrace\det \left( I_n + \Wt (\Wt^T \Wt + I_p)^{-1} \Wt^T\right) \right \rbrace ^{1/2} \\
\text{ and } C_1'' &= C_1/C_1'.
\end{align*} 
The last inequality follows from the fact the integrand is a normal density. \\

\noindent Therefore, from $(\ref{I_Azeta})$ , $(\ref{milsubd})$ and $(\ref{innrint})$ we get
\begin{align*}\label{I_A_ubd}
I_{A_\zeta} &\leq C_1''  \int_{A_\zeta} \exp \left[ -\frac{1}{2} \left\lbrace \thetab^T(2\Wt^T \Wt + I_p)\thetab - 2 \thetab^T \vt \right\rbrace \right] \widetilde{Q}(\ab_2) \exp \left[ \frac{1}{2} \ab_2^T \ab_2 \right] \\
&\qquad\qquad \times \exp \left[ \frac{1}{2} \left(\frac{4}{1-\epsilon}\right)  \right. \\ 
&\qquad \qquad \qquad \qquad \times  \left. \ab_1^T  \left( I_n + \Wt (\Wt^T \Wt + I_p)^{-1} \Wt^T\right)^{-1} \ab_1 \right]  \:d\thetab \\
&= C_1''  \int_{A_\zeta} \widetilde{Q}(\ab_2) \\
&\qquad \qquad \times \exp \left[ -\frac{1}{2} \left\lbrace \thetab^T(2\Wt^T \Wt + I_p)\thetab - \ab_2^T\ab_2  - 2 \thetab^T \vt \right. \right. \\
&\qquad \qquad \qquad \left.\left. - \left(\frac{4}{1-\epsilon}\right) \ab_1^T  \left( I_n + \Wt (2\Wt^T \Wt + I_p)^{-1} \Wt^T\right)^{-1} \ab_1  \right\rbrace \right]  \:d\thetab \\
&=C_1''  \int_{A_\zeta} \widetilde{Q}(\ab_2) \: \exp \left[ -\frac{1}{2} \left\lbrace G(\thetab,\epsilon) -  2 \thetab^T \vt \right\rbrace \right] \:d\thetab \numbereqn
\end{align*}
where
\begin{align*}\label{g_theta}
G(\thetab,\epsilon) &= \thetab^T(2\Wt^T \Wt + I_p)\thetab - \ab_2^T\ab_2 \\ 
& \qquad - \left(\frac{4}{1-\epsilon}\right) \ab_1^T  \left( I_n + \Wt (\Wt^T \Wt + I_p)^{-1} \Wt^T\right)^{-1} \ab_1 \\
&= 2(\ab_1 + \ab_2)^T(\ab_1 + \ab_2) + \thetab^T\thetab - \ab_2^T\ab_2 \\
& \qquad - \left(\frac{4}{1-\epsilon}\right) \ab_1^T  \left( I_n + \Wt (\Wt^T \Wt + I_p)^{-1} \Wt^T\right)^{-1} \ab_1 \\
&= 2\ab_1^T\ab_1 - \left(\frac{4}{1-\epsilon}\right) \ab_1^T  \left( I_n + \Wt (\Wt^T \Wt + I_p)^{-1} \Wt^T\right)^{-1} \ab_1  \\
& \qquad + \ab_2^T\ab_2 + \thetab^T\thetab \numbereqn
\end{align*}
the last equality following from the fact that $\ab_1^T\ab_2 = 0$. Therefore, to prove 
$(\ref{I_a_cond})$ it would be sufficient to show that for some $\epsilon \in (0,1)$
\begin{align}\label{integrand}
\widetilde{Q}(\ab_2) \: \exp \left[ -\frac{1}{2} \left\lbrace G(\thetab,\epsilon) -  2 \thetab^T \vt \right\rbrace \right]
\end{align}
is integrable on $A_\zeta$. This holds when $G(\thetab,\epsilon)$  is a positive definite quadratic form in $\thetab$ on $A_\zeta$, as then, for some sufficiently large 
$C_2 > 0$, $G(\thetab,\epsilon) - 2 \thetab^T \vt + C_2$ is also positive definite, making the exponential term in $(\ref{integrand})$ a constant multiple of an appropriate 
multivariate normal density and the integrability of $(\ref{integrand})$ follows from the existence of (positive) moments of any multivariate normal distribution. 

Therefore, our objective is to show that there exists an $\epsilon \in (0,1)$ for which 
$G(\thetab, \epsilon)$ is a positive definite quadratic form (on $A_\zeta$) in $\thetab$ 
when at least one of \ref{trcond1} and \ref{trcond2} holds. Note that on $A_\zeta$, 
each entry of $\ab_1$ and $\ab_2$ is a linear function of $\thetab$. It follows from 
(\ref{g_theta}) that on $A_\zeta$, $G(\thetab, \epsilon)$ is a quadratic form in 
$\thetab$ for every $\epsilon > 0$. Since $\zeta$ is arbitrarily chosen, to achieve our 
objective, it is enough to show that for some $\epsilon \in (0,1)$, $G(\thetab, \epsilon) 
> 0$ for every $\thetab \in \R^p$. Since $X$ (and hence $\Wt$) is assumed to have 
full column rank if $n \geq p$ and full row rank if $n < p$, it follows that either $\Wt^T 
\Wt$ (when $n \geq p$) or $\Wt \Wt^T$ (when $n < p$) is invertible. Therefore, when 
$n \geq p$
\[
\thetab^T\thetab = \thetab^T \Wt^T\Wt (\Wt^T\Wt)^{-2} \Wt^T\Wt \thetab = (\ab_1 + \ab_2)^T \Wt (\Wt^T\Wt)^{-2} \Wt^T (\ab_1 + \ab_2)  
\]
and when $n < p$
\[
\thetab^T\thetab \geq \thetab^T \proj_{W^T} \thetab = \thetab^T \Wt^T (\Wt\Wt^T)^{-1} \Wt\thetab = (\ab_1 + \ab_2)^T (\Wt\Wt^T)^{-1} (\ab_1 + \ab_2)
\]
where for any matrix $B$, $\proj_B$ denotes the orthogonal projection (matrix) onto the column space of $B$, and the  inequality follows from the fact that  $\xb^T\xb \geq \xb^T \proj_B \xb$, for any $\xb \in \R^{k}$, $k$ being the number of rows of $B$. Thus, it follows that letting 
\[
M =
\begin{cases}
\Wt (\Wt^T\Wt)^{-2} \Wt^T & \text{if } n \geq p \\
 (\Wt\Wt^T)^{-1} & \text{if } n < p
\end{cases}
\]
yields, in general,
\[
\thetab^T\thetab \geq  (\ab_1 + \ab_2)^T M (\ab_1 + \ab_2). 
\]

\noindent Hence, from (\ref{g_theta}) 
\begin{align*} \label{H(a)defn}
G(\thetab, \epsilon)  & \geq 2\ab_1^T\ab_1 - \left(\frac{4}{1-\epsilon}\right) \ab_1^T  \left( I_n + \Wt (\Wt^T \Wt + I_p)^{-1} \Wt^T\right)^{-1} \ab_1 \\
& \qquad + \ab_2^T\ab_2 + (\ab_1 + \ab_2)^T M (\ab_1 + \ab_2)\\
&=  2\ab_1^T\ab_1 - 4 \ab_1^T  \left( I_n + \Wt (\Wt^T \Wt + I_p)^{-1} \Wt^T\right)^{-1} \ab_1 \\
& \qquad + \ab_2^T\ab_2 + (\ab_1 + \ab_2)^T M (\ab_1 + \ab_2)\\
&\qquad - 4 \left(\frac{\epsilon}{1-\epsilon}\right) \ab_1^T  \left( I_n + \Wt (\Wt^T \Wt + I_p)^{-1} \Wt^T\right)^{-1} \ab_1 \\
&=: H(\ab,\epsilon)\numbereqn 
\end{align*}
where $\ab^T = (\ab_1^T,\ab_2^T)$. From (\ref{H(a)defn}) it follows that  in order to prove $G(\thetab, \epsilon) > 0$ for all $\thetab \in \R^p$ it is enough to show that $H(\ab,\epsilon) > 0$ for all $\ab \in \R^n$.\\

\noindent Now letting
\[
\underset{n \times n}{R} = 
\begin{pmatrix}
2I_n + M - 4 \left( I_n + \Wt (\Wt^T \Wt + I_p)^{-1} \Wt^T\right)^{-1} & M\\
M & I_n + M
\end{pmatrix}
\]
and
\[
\underset{n \times n}{S} = 
\begin{pmatrix}
\left( I_n + \Wt (\Wt^T \Wt + I_p)^{-1} \Wt^T\right)^{-1} & \underset{n \times n}{0} \\
\underset{n \times n}{0} & \underset{n \times n}{0}
\end{pmatrix}
\]
yields
\begin{align*}
H(\ab,\epsilon) = \ab^T R \ab -  4 \left(\frac{\epsilon}{1-\epsilon}\right) \ab^T S \ab
\end{align*}
where $0_{n \times n}$ denotes an $n \times n$ matrix with all elements equal to zero.\\

\noindent
Note that $S$ is positive semi-definite. Hence, if $H^*(\ab) := \ab^T R \ab =  H(\ab,0)$ is positive definite, then it follows by Proposition~\ref{prop3} that $H(\ab, \epsilon)$ is 
positive definite in $\ab$ for sufficiently small $\epsilon$. Thus, our objective boils down  in showing that when at least one of \ref{trcond1} and \ref{trcond2} holds, 
$H^*(\ab)$ is positive definite in $\ab$. We shall prove this fact by considering the cases $n \geq p$ and $n < p$ separately.

\subsection*{Case I : $n\geq p$}

Consider the following singular value decomposition.
\begin{align}
\underset{n \times p}{\Wt} = \underset{ n \times p}{ U} \; \underset{ p \times p}{ D} \; {\underset{ p \times p}{ V}}^T
\end{align}
where $V \in \R^{p \times p}$ is orthogonal, $D \in \R^{p \times p}$ is diagonal, say $D = \diag(d_1, \cdots, d_p)$ with  $d_i \neq 0$ for all $i = 1,\cdots,n$, and $U \in \R^{n \times p}$ is a matrix with orthogonal columns. Further, let 
\[
\underset{n \times n}{U^*} = \left( \underset{n \times p}{U} \left| \underset{n \times (n - p)}{\Ut} \right. \right) \
\]
be orthogonal in $\R^{n \times n}$.\\

\noindent Then,
\begin{align*}
\left( I_n + \Wt (\Wt^T \Wt + I_p)^{-1} \Wt^T\right)^{-1} &=  U \left(\frac{D^2 + \tau I_p}{2D^2 + \tau I_p}\right) U^T + \Ut\Ut^T 
\end{align*}
and
\[
M = \Wt (\Wt^T\Wt)^{-2} \Wt^T = U D^{-2} U^T = U \left(\frac{ I_p}{D^2 }\right)U^T
\]
where, for diagonal matrices 
\[
\tilde{N}_{k \times k} = \diag(\tilde{n}_1, \cdots,\tilde{n}_k) 
\text{ and } 
N_{k \times k} =\diag(n_1, \cdots, n_k) 
\]
with $n_i \neq 0$ for all $i = 1,\cdots,k$, we define
\[
\frac{\tilde{N}}{N} := \diag\left( \frac{\tilde{n}_1}{n_1}, \cdots, \frac{\tilde{n}_k}{n_k} \right).
\]

\noindent Then,
\[
\ab_1^T \ab_2  = \ab_1^T(UU^T + \Ut\Ut^T)\ab_2 = 0 \implies \ab_1^TUU^T\ab_2 = - \ab_1^T\Ut\Ut^T\ab_2
\]
and
\[
\Ut^T(\ab_1+\ab_2) = \Ut^T \Wt \thetab = \Ut^T U D V^T \thetab = \bm{0} \implies \Ut^T \ab_1 = - \Ut^T \ab_2
\]
which means
\[
\ab_1UU^T\ab_2 = - \ab_1 \Ut \Ut^T\ab_2 = \ab_1 \Ut \Ut^T\ab_1 = \ab_2 \Ut \Ut^T\ab_2  
\]
and hence
\begin{align*} \label{H*1}
 H^*(\ab) &= 2\ab_1^T\ab_1 - 4 \ab_1^T  \left( I_n + \Wt (\Wt^T \Wt + I_p)^{-1} \Wt^T\right)^{-1} \ab_1 + \ab_2^T\ab_2 \\
 & \qquad + (\ab_1 + \ab_2)^T M (\ab_1 + \ab_2)\\
 &= 2\ab_1^T U U^T \ab_1 + 2\ab_1^T \Ut \Ut^T \ab_1 - 4 \ab_1^T U \left(\frac{D^2 + \tau I_p}{2D^2 + \tau I_p}\right) U^T \ab_1  \\
 &\qquad - 4 \ab_1^T \Ut\Ut^T \ab_1 + \ab_2^T U U^T \ab_2 + \ab_2^T \Ut \Ut^T \ab_2 \\
 &\qquad + \ab_1^T  U \left(\frac{ I_p}{D^2 }\right)U^T \ab_1 + \ab_2^T  U \left(\frac{ I_p}{D^2 }\right)U^T \ab_2 + 2 \ab_1^T  U \left(\frac{ I_p}{D^2 }\right)U^T \ab_2\\
 &= 2\ab_1^T U U^T \ab_1 + 2\ab_1UU^T\ab_2 - 4 \ab_1^T U \left(\frac{D^2 + \tau I_p}{2D^2 + \tau I_p}\right) U^T \ab_1  \\
 &\qquad - 4 \ab_1UU^T\ab_2 + \ab_2^T U U^T \ab_2 + \ab_1UU^T\ab_2 \\
 &\qquad + \ab_1^T  U \left(\frac{ I_p}{D^2 }\right)U^T \ab_1 + \ab_2^T  U \left(\frac{ I_p}{D^2 }\right)U^T \ab_2 + 2 \ab_1^T  U \left(\frac{ I_p}{D^2 }\right)U^T \ab_2\\
  &= \ab_1^T U \left(2I_p + \frac{ I_p}{D^2 } - \frac{4 D^2 + 4 I_p}{2D^2 + \tau I_p} \right) U^T \ab_1 + \ab_2^TU \left(I_p + \frac{ I_p}{D^2 } \right) U^T \ab_2 \\
  &\qquad +  2 \ab_1^T U \left(\frac{ I_p}{D^2 } - \frac12  I_p \right) U^T \ab_2\\
  &= \at_1^T  \left(2I_p + \frac{ I_p}{D^2 } - \frac{4 D^2 + 4 I_p}{2D^2 + \tau I_p} \right) \at_1 + \at_2^T \left(I_p + \frac{ I_p}{D^2 } \right) \at_2 \\
  &\qquad + 2 \at_1^T \left(\frac{ I_p}{D^2 } - \frac12  I_p \right) \at_2 \numbereqn
\end{align*}
where $\at_j = U^T \ab_j$ for $j = 1,2$.\\

\noindent Note that, 
\begin{align}\label{A11}
2I_p + \frac{ I_p}{D^2 } - \frac{4 D^2 + 4 I_p}{2D^2 + \tau I_p} = \frac{ I_p}{D^2 } - \frac{2 I_p}{2D^2 + \tau I_p} = \frac{ I_p}{D^2(2D^2 + \tau I_p)} > 0_{p \times p}  
\end{align}
and
\begin{align}\label{A22}
I_p + \frac{ I_p}{D^2 } > 0_{p \times p}
\end{align}
where for two symmetric matrices $A$ and $B$ of the same order, $A>B$ means $A-B$ is positive definite. This shows that the first two terms in (\ref{H*1}) are strictly positive. \\

\noindent Now, under \ref{trcond2},
\begin{align*}
& X Q^{-1/2} \text{ is rectangular diagonal} \\
\implies & \Wt = W Q^{-1/2} \text{ is rectangular diagonal} \\
\implies & U^* = I_n, V = I_p \\
\implies & \at_1 ^T D^* \at_2 \; = \ab_1^TUD^*U^T\ab_2 + \ab_1^T\Ut D^* \Ut ^T\ab_2 \\
& \qquad \qquad =\ab_1^T
\begin{pmatrix}
{D^*}_{11}^{p \times p} & 0^{p \times (n-p)} \\
0^{(n-p) \times p} & 0^{(n-p) \times (n-p)}  
\end{pmatrix}
\ab_2 \\
& \qquad \qquad \quad + \ab_1^T
\begin{pmatrix}
0^{p \times p} & 0^{p \times (n-p)} \\
0^{(n-p) \times p} & {D^*}_{22}^{(n-p) \times (n-p)}  
\end{pmatrix}
\ab_2\\
& \qquad \qquad = 0, 
\end{align*}
for any diagonal matrix 
\[
D^*_{n \times n} = 
\begin{pmatrix}
D^*_{11} & 0\\
0 & {D^*}_{22}  
\end{pmatrix}.
\] 
The last equality follows from the fact that $a_{1i} a_{2i} = 0$ for all $i=1,\cdots,n$. Therefore letting $D^* = \frac{ I_p}{D^2 } - \frac12  I_p$ makes the cross product term in (\ref{H*1}) equal to zero, which means, under \ref{trcond2}, $H^*(\ab)$ is a sum of two positive quantities, and hence is strictly positive.\\

\noindent Again, note that $Q^{-1/2}X^TX Q^{-1/2} = Q^{-1/2}W^TW Q^{-1/2} = \Wt^T\Wt =  V D^2 V^T$. Hence, the eigenvalues of $Q^{-1/2}X^TX Q^{-1/2}$ are $d_1^2, \dots, d_p^2$. Let $k_i = d_i^2/2$ for all $i$. Therefore, under \ref{trcond1}, for all $i = 1,\cdots,p$,
\begin{align*}
& d_i^2 < \frac{7}{2} \implies k_i < \frac{7}{4} \implies 7k_i^2 > 4k_i^3 \\
\implies & 1 + 2k_i + 7k_i^2 > 1 + 2k_i + 4k_i^3\\
\implies & 1 + 2k_i >  1 +2k_i -7k_i^2 + 4k_i^3 = (1-k_i)^2 (1+4k_i) \\
\implies & 1 +d_i^2 > \left(1 - \frac{d_i^2 }{2}\right)^2  (2d_i^2 + 1) \\
\implies & 1 + \frac{ 1}{d_i^2} > \frac{ 1}{d_i^4 }\left(1 - \frac{d_i^2 }{2}\right)^2 d_i^2 \: (2d_i^2 + 1) 
= \left(\frac{ 1}{d_i^2 } - \frac12   \right)^2 d_i^2 \:(2d_i^2 + 1). 
\end{align*}

\noindent This implies,
\begin{align*}\label{D1}
\left(I_p + \frac{ I_p}{D^2 } \right) &> \left(\frac{ I_p}{D^2 } - \frac12  I_p \right)^2 D^2(2D^2 + \tau I_p) \\
&= \left(\frac{ I_p}{D^2 } - \frac12  I_p \right) D^2(2D^2 + \tau I_p) \left(\frac{ I_p}{D^2 } - \frac12  I_p \right) \\
&= \left(\frac{ I_p}{D^2 } - \frac12  I_p \right) \left(2I_p + \frac{ I_p}{D^2 } - \frac{4 D^2 + 4 I_p}{2D^2 + \tau I_p} \right)^{-1} \left(\frac{ I_p}{D^2 } - \frac12  I_p \right) \numbereqn
\end{align*}
Combining (\ref{H*1}), (\ref{A11}), (\ref{A22}) and (\ref{D1}), it follows that $H^*(\ab)$ is positive definite.

\subsection*{Case II : $n < p$}
We slightly abuse our notation by considering the following singular value decomposition:
\begin{align}
\underset{n \times p}{\Wt^T} = \underset{ p \times n}{ U} \; \underset{ n \times n}{ D} \; {\underset{ n \times n}{ V}}^T
\end{align}
where as before (but now with different dimensions) $V \in \R^{n \times n}$ is orthogonal, $D \in \R^{n \times n}$ is diagonal, say $D = \diag(d_1, \cdots, d_n)$ where no $d_i$ is equal to zero, and $U \in \R^{p \times n}$ is a matrix with orthogonal columns and
\[
\underset{p \times p}{U^*} = \left( \underset{p \times n}{U} \left| \underset{p \times (p - n)}{\Ut} \right. \right) 
\]
is orthogonal in $\R^{p \times p}$. Here
\begin{align*}
\left( I_n + \Wt (\Wt^T \Wt + I_p)^{-1} \Wt^T\right)^{-1} &=  V \left(\frac{D^2 + \tau I_n}{2D^2 + \tau I_n}\right) V^T 
\end{align*}
and
\[
M = \Wt^T (\Wt\Wt^T)^{-1} \Wt = V D^{-2} V^T = V \left(\frac{ I_n}{D^2 }\right) V^T.
\]

\noindent
Hence
\begin{align*} \label{H*2}
 H^*(\ab) &= 2\ab_1^T\ab_1 - 4 \ab_1^T  \left( I_n + \Wt (\Wt^T \Wt + I_p)^{-1} \Wt^T\right)^{-1} \ab_1 + \ab_2^T\ab_2 \\
 & \qquad + (\ab_1 + \ab_2)^T M (\ab_1 + \ab_2)\\
 &= 2\ab_1^T V V^T \ab_1  - 4 \ab_1^T V \left(\frac{D^2 + \tau I_n}{2D^2 + \tau I_n}\right) V^T \ab_1 + \ab_2^T V V^T \ab_2   \\
 &\qquad  + \ab_1^T  V \left(\frac{ I_n}{D^2 }\right)V^T \ab_1 + \ab_2^T  V \left(\frac{ I_n}{D^2 }\right)V^T \ab_2 + 2 \ab_1^T  V \left(\frac{ I_n}{D^2 }\right)V^T \ab_2\\
 &= \ab_1^T V \left(2I_n + \frac{ I_n}{D^2 } - \frac{4 D^2 + 4 I_n}{2D^2 + \tau I_n} \right) V^T \ab_1 + \ab_2^T V \left(I_n + \frac{ I_n}{D^2 } \right) V^T \ab_2 \\
 &\qquad + 2 \ab_1^T V \left(\frac{ I_n}{D^2 } - \frac12 I_n \right) V^T \ab_2\\
 &=\at_1^T  \left(2I_n + \frac{ I_n}{D^2 } - \frac{4 D^2 + 4 I_n}{2D^2 + \tau I_n} \right) \at_1 + \at_2^T \left(I_n + \frac{ I_n}{D^2 } \right) \at_2  \\
 &\qquad + 2 \at_1^T  \left(\frac{ I_n}{D^2 } - \frac12  I_n \right) \at_2 \numbereqn
\end{align*}
where $\at_j = V^T \ab_j$ for $j = 1,2$; and the equality in the second last line arises from the fact that $\ab_1^TVV^T\ab_2 = \ab_1^T\ab_2 = 0$. Notice the similarities between (\ref{H*2}) and (\ref{H*1}) and note that the non-zero eigenvalues of 
\[
Q^{-1/2}X^TX Q^{-1/2} = \Wt^T\Wt = U D^2 U^T \text{ and }  \Wt \Wt^T = V D^2 V^T
\] 
are the same, namely  $d_1^2,\dots,d_n^2$. Therefore by exactly similar arguments as provided in the previous case, it follows that in this case also, $H^*(\ab)$ is positive definite if either \ref{trcond1} or  \ref{trcond2}   holds.\\  

\noindent Thus, both when $n \geq p$ and $n < p$, if at least one of  \ref{trcond1}  and   \ref{trcond2} holds, $H^*(\ab)$ is  positive definite in $\ab$. As mentioned previously, this ensures integrability of $I_{A_\zeta}$ as given in (\ref{I_Azeta}). Since $\zeta \subseteq \{1,2, \cdots, n\}$ is chosen arbitrarily, it follows that $\Psi$ has the trace-class property. 
\end{proof}

\begin{remark}
Since the positive eigenvalues of the matrices $\Wt^T\Wt$ and $\Wt \Wt^T$ in the proof of Theorem~\ref{thm2} are the same, condition   \ref{trcond1}   can be equivalently expressed as the following.
\begin{enumerate}[label = (A\arabic*)]
\item  \label{trcond1dash} \emph{All eigenvalues (or all non-zero eigenvalues, if $n \geq p$) of $XQ^{-1} X^T$ are less than $7/2$.}
\end{enumerate} 
\end{remark}

\begin{remark} \label{remprior}
The prior considered in this paper reduces to an approximate flat prior when $Q$ is ``small'' (approaching the zero matrix). However, when $Q$ is ``small'', $Q^{-1}$  is 
``large''; which makes the (positive) eigenvalues of $XQ^{-1} X^T$ large. It follows that, when $Q$ is so small that at least one eigenvalue of $XQ^{-1} X^T$ is bigger 
than or equal to 7/2, condition \ref{trcond1dash} gets violated, and  Theorem~\ref{thm2} can no longer be applied.  
\end{remark}

\begin{remark}\label{gprior}
When $n \geq p$ and $X$ has full column rank, Zellner \cite{zellner:1983} specifies a 
Gaussian prior distribution for $\betab$ with the prior covariance matrix having the 
form $Q^{-1} = g(X^TX)^{-1}$, where $g$ is a positive scaling constant. This prior is 
commonly referred to as {\it Zellner's $g$-prior}. Under this prior 
\[
Q^{-1/2}X^TX Q^{-1/2} = g \left(X^TX\right)^{-1/2} X^TX \left(X^TX\right)^{-1/2} = g I_p 
\]
has eigenvalue $g$ with multiplicity $p$. Hence condition \ref{trcond1} is satisfied as long as $g < 7/2$, or equivalently $g^{-1} > 2/7$. Thus under this prior, a sufficient condition for  the AC-DA Markov chain $\Psi$  to be trace-class is $g < 7/2$.\\
\end{remark}

\section{Sandwich Algorithms} \label{sandalg}
As mentioned previously, one of the common problems with DA algorithms is that they are slow to converge.  However, significant improvements over the convergence rate of a two block DA Makrov chain can be achieved  by using a so-called \emph{sandwich algorithm}, where one simple and computationally inexpensive intermediate step  ``sandwiched'' between the two steps of the DA algorithm is added at each iteration (see e.g., Liu and Wu \cite{liu:wu:1999}; Meng and van Dyk \cite{Meng:vanDyk:1999}; Hobert and Marchev \cite{hobert:marchev:2008}).  Consider our AC-DA Markov chain $\Psi$ once again and let $\betab$ be its current state. One iteration of a sandwich algorithm corresponding to  $\Psi$  comprises of the following three (instead of two, as in the AC-DA) steps. The first step is similar to AC-DA in the sense that a (latent) random variable $\zb \sim \pi(\zb|\betab,\yb)$ is generated. Next, a Markov transition function $R$  that is  reversible with respect to $\pi(\zb|\yb)\: d\zb$ (i.e., $R(\zb, d\zb') \: \pi(\zb|\yb) \: d\zb = R(\zb', d\zb) \: \pi(\zb'|\yb) \: d\zb'$) is considered, where $\pi(\zb|\yb)$ denotes the $\zb$-marginal of $\pi(\betab,\zb|\yb)$ in (\ref{pi_betaz|y}). The intermediate second step for the sandwich algorithm then amounts to generating a random variable $\zb'$ from the measure $R(\zb,\cdot)$.  The third and final step in the sandwich algorithm is again similar to the last step in AC-DA except for the fact that here, instead of  $\zb$, $\zb'$  is used. That is, the third step in the sandwich algorithm entails generating the next state $\betab'$  from $\pi(\betab|\zb',\yb)$. The intermediate step  involving the generation of $\zb'$ from $\zb$ is typically done with the help of a low (generally one or two) dimensional random variable, making the DA and the sandwich algorithm comparable in terms of computational efficiency.

In order to make precise comparisons between the DA and the sandwich algorithms, we first need to introduce some notations. Let $\Psit$ be the Markov chain obtained by the sandwich algorithm. Analogous to the notations used in Section~\ref{geoergodicity} and \ref{trace-class}, let $\Kt$ denote the Markov operator associated with the sandwich algorithm, i.e., for all $h \in L^2_0(\pi)$, $\Kt$ maps $h$ to
\[
(\Kt h) (\betab) := \int_{\R^p} h(\betab') \: \kt(\betab,\betab') \:d\betab'
\]  
where $\kt$, the Markov transition density of $\Psit$, is defined as follows: 
\[
\kt(\betab,\betab') = \int_{\mathcal{Z}} \int_{\mathcal{Z}} \pi(\betab'|\zb',\yb) R(\zb, d\zb') \pi(\zb|\betab,\yb) \:d\zb.
\]

A sandwich algorithm is \emph{always} at least as good as the DA algorithm in the sense of having a smaller operator norm, that is, we always have $\| \Kt\| \leq \| K\|$, though a strict inequality may not hold in general. Here $K$ denotes the Markov operator associated with the corresponding DA Markov chain. Note that if a DA Markov chain is geometrically ergodic, then so is the sandwich Markov chain due to the relationship $\|K\| \leq \|K\| <1$. (Recall that a reversible Markov chain is geometrically ergodic if and only if the corresponding operator $K$ satisfies $\|K\| < 1$ (Roberts and Rosenthal \cite{roberts:rosenthal:1997}).) Thus, as long as the original DA algorithm is geometrically ergodic, a CLT holds for the sandwich algorithm as well. In particular, let $g \in L^2(\pi(\betab|\yb))$ such that $E_\pi g(\betab)^{2} < \infty$, and let $(\betab_N)_{N=1}^m$ and $(\betat_N)_{N=1}^m$ respectively denote the observations generated by the DA and the sandwich algorithm. Define $\bar{g}_m := m^{-1} \sum_{N=1}^m g\left(\betab_N\right)$ and $\widetilde{g}_m := m^{-1} \sum_{N=1}^m g\left(\betat_N\right)$. Then there exist positive, finite quantities $\sigma^2_{g}$ and $\sigmat^2_{g}$ such that, as $m \rightarrow \infty$,
\begin{align*}
\sqrt{m}(\bar{g}_m - E_\pi g) & \xrightarrow{d} N(0, \sigma^2_{g})\\
\text{and } \sqrt{m}(\widetilde{g}_m - E_\pi g) & \xrightarrow{d} N(0, \sigmat^2_{g}).
\end{align*}
Moreover Hobert and Marchev \cite[Theorem 4]{hobert:marchev:2008} show that $\sigma^2_{g} \leq \sigmat^2_{g}$, that is, by using a sandwich algorithm, one gets the asymptotic variance of $\widetilde{g}_m$ no larger (possibly smaller) than that of $\bar{g}_m$.

One class of sandwich algorithms, the so called Parameter Expanded Data Augmentation (PX-DA) algorithms (Liu and Wu \cite{liu:wu:1999}, Meng and Van Dyk \cite{Meng:vanDyk:1999}), use a proper probability measure for the Markov transition function $R$ in the intermediate step. While all PX-DA algorithms are aimed at improving the original DA algorithm, following Hobert and Marchev \cite{hobert:marchev:2008}, one can get a sandwich algorithm that is uniformly better than all PX-DA algorithms, as long as a certain group structure is present in the problem. This ``best'' PX-DA algorithm, while technically not a PX-DA itself as it does not use a proper probability measure for $R$, and rather involves Haar measure, is called the Haar PX-DA algorithm. We now describe the form of the Haar PX-DA algorithm  corresponding to the AC-DA algorithm.

Using the similar notations as in Hobert and Marchev \cite{hobert:marchev:2008}, let $\G$ be the multiplicative group  $(\R_+, \circ)$ where the group composition $\circ$ is defined as multiplication, i.e., for all $g_1, g_2 \in \G$,   $g_1 \circ g_2 = g_1 g_2$. $\G$ has $e = 1$ as its identity element and $g^{-1} = 1/g$. The multiplicative group  $(\R_+, \circ)$ is uni-modular with Haar measure $\mu_l(dg) = dg/g$, $dg$ being the usual Lebesgue measure on $\R_+$. Recall that $\Z$ denotes the support of the conditional density $\pi(\zb|\yb)$ of $\zb$ given $\yb$. (In particular, $\Z$ is the Cartesian product of $n$ half lines $\R_+$ or $\R_-$ according as $y_i = 1$ or $0$.) Let us define a (left) group action of $\G$ on $\Z$,  which act through component-wise multiplication, i.e., $g\in \G, \zb = (z_1,\cdots,z_n)^T \in \Z \implies g\zb = (gz_1, \cdots, gz_n)^T$. With this (left) group action, the Lebesgue measure on $\R^n$ is relatively left invariant with multiplier $\chi(g) = g^n$; i.e., for all $g \in \G$ and all integrable functions $h:\R^n \rightarrow \R$,
\[
g^n \int_{\R^n} h(g\zb)\:d\zb = \int_{\R^n} h(\zb)\:d\zb.
\]   
Then the intermediate step (that involves drawing $\zb'$ from $\zb$ using some Markov transition function $R$) of the Haar PX-DA algorithm amounts to generating a random variable $g$ from a density proportional to
\[
\chi(g) \: \pi(g\zb|\yb) \: \mu_l(dg) = g^{n-1} \:  \pi(g\zb|\yb) \: dg =:w(g) \: dg
\]
and defining $\zb' = g \zb = (gz_1,\cdots,gz_n)^T$. Straightforward calculations show that the $\zb$-marginal of the joint density in (\ref{pi_betaz|y}) satisfies 
\begin{align*}\label{pi_z|y}
\pi(\zb|\yb)  &\propto \prod_{i=1}^n \left \lbrace \left(\one_{(0,\infty)}(z_i)\right) ^{y_i}
												 \left(\one_{(-\infty,0]}(z_i) \right) ^{1-y_i} \right\rbrace \\
&\quad	\times \exp \left[ -\frac{1}{2} \left\lbrace \zb^T\left( I_n - X(X^TX+Q)^{-1}X^T\right)\zb \right. \right. \\
& \qquad \qquad - \left. \left. 2 \zb^T X(X^TX+Q)^{-1} \vb \right\rbrace \right] \numbereqn			  
\end{align*}
so that 
\begin{align*}
\pi(g\zb|\yb)  &\propto \prod_{i=1}^n \left \lbrace \left(\one_{(0,\infty)}(g z_i)\right) ^{y_i}
										 \left(\one_{(-\infty,0]}(g z_i) \right) ^{1-y_i} \right\rbrace  \\
&\quad	\times \exp \left[ -\frac{1}{2} \left\lbrace g^2 \zb^T\left( I_n - X(X^TX+Q)^{-1}X^T\right)\zb 
 \right. \right. \\
& \qquad \qquad  \qquad \quad - \left. \left. 2 g \zb^T X(X^TX+Q)^{-1} \vb \right\rbrace \right] \\
&=\prod_{i=1}^n \left \lbrace \left(\one_{(0,\infty)}(z_i)\right) ^{y_i} \left(\one_{(-\infty,0]}(z_i) \right) ^{1-y_i} \right\rbrace \\
&\qquad \times \exp \left[ -\frac{1}{2} \left\lbrace g^2 A (\zb) - 2 g B(\zb) \right\rbrace \right]
\end{align*}
where
\begin{align}\label{AzBz}
\begin{cases}
A (\zb) &= \quad  \zb^T\left( I_n - X(X^TX+Q)^{-1}X^T\right)\zb \\
B(\zb) &= \quad \zb^T X(X^TX+Q)^{-1} \vb.
\end{cases}
\end{align}

\noindent Hence, $w(g) \: dg$ in the intermediate step for the Haar PX-DA algorithm satisfies
\begin{align*}\label{w*g}
w(g) \: dg & \propto g^{n-1} \prod_{i=1}^n \left \lbrace \left(\one_{(0,\infty)}(z_i)\right) ^{y_i} \left(\one_{(-\infty,0]}(z_i) \right) ^{1-y_i} \right\rbrace \\
&\qquad \times \exp \left[ -\frac{1}{2} \left\lbrace g^2 A (\zb) - 2 g B(\zb) \right\rbrace \right] dg\\
&\propto  g^{n-1} \exp \left[ -\frac{1}{2} \left\lbrace g^2 A (\zb)  2 g B(\zb) \right\rbrace \right] dg
=: w^*(g) \: dg. \numbereqn
\end{align*}
Note that
\[
 I_n - X(X^TX+Q)^{-1}X^T = I_n - \Xt(\Xt^T\Xt + I_p)^{-1}\Xt^T
\]
where  $\Xt = X Q^{-1/2}$, and it follows from Proposition~\ref{prop2} in Appendix~\ref{appB} that the right hand side is positive definite. This implies $A(\zb)$ is strictly positive for any non zero $\zb$, and hence we can indeed find a density $\tilde{w}^*(g) = w^*(g)/\int_0^\infty w^*(s)\:ds$ proportional to $w^*(g)$. Therefore, using (\ref{w*g}), transition  from $\betab_m$ to $\betab_{m+1}$ of the Haar PX-DA  Markov chain $\Psi_H$  is obtained as follows.\\

\noindent
\begin{minipage}{\textwidth}

\hlinehere

$(m+1)$st iteration for the Haar PX-DA Markov chain $\Psi_H$

\hlinehere

\begin{enumerate}[label = (\roman*)]
\item Draw independent $z_1, \cdots, z_n$ with 
\[
z_i \sim \TN \left( \xb_i^T \betab_m , 1 , y_i \right),\; i = 1,\cdots,n
\]
and call  $\zb = (z_1,\cdots,z_n)^T$.

\item Draw $g$ from a density proportional to
\[
w^*(g) \: dg =  g^{n-1} \exp \left[ -\frac{1}{2} \left\lbrace g^2 A (\zb) - 2 g B(\zb) \right\rbrace \right] dg 
\]
where $A(\zb)$ and $B(\zb)$ are as given in (\ref{AzBz}) and call $\zb' = g \zb = (gz_1,\cdots,gz_n)^T$.

\item Draw $\betab_{m+1} \sim \N_p \left( \left(X^TX + Q\right)^{-1}\left(\vb + X^T\zb' \right), \left(X^TX + Q\right)^{-1} \right)$.
\end{enumerate}

\hlinehere \\

\end{minipage}

It can be easily seen that the conditional posterior density $\pi(\betab|\zb,\yb)$ is not invariant under the group action of $\G$ on $\Z$, i.e., $\pi(\betab|\zb,\yb) = \pi(\betab|g\zb,\yb)$ does not hold in general for $g \in \G$ (except, of course, the identity element). When the AC-DA Markov chain $\Psi$ is trace-class, using the results in Khare and Hobert \cite{khare:hobert:2011}, it follows that the Haar PX-DA chain $\Psi_H$ is also trace-class. Furthermore $\Psi_H$ is strictly better than $\Psi$ in the following sense. Let $(\lambda_i)_{i=0}^\infty$ and $({\lambda_H}_i)_{i=0}^\infty$ denote the non-increasing sequences of eigenvalues corresponding to $\Psi$ and $\Psi_H$ respectively. Then ${\lambda_H}_i \leq {\lambda}_i$ for every $i \geq 0$, with at least one strict inequality.

\begin{remark}\label{B=0}
Note that when $B(\zb)$ in (\ref{AzBz}) is zero (which is the case when the prior mean $Q^{-1}\vb$ is $\bm0 \iff \vb = \bm0$), $\tilde{w}^*$ reduces to  a square gamma density, i.e., the density corresponding to a random variable whose square follows a gamma distribution. Since generating observations from univariate gamma distributions is simple, inexpensive and implemented in virtually every statistical package, when $B(\zb) = 0$, additional costs due to the extra steps in $\Psi_H$ become essentially negligible. When $B(\zb) \neq 0$, $\tilde{w}^*(g)$ no longer remains a square gamma density (or any standard density, for that matter). However, it is still possible to generate observations  from $\tilde{w}^*(g)$ by rejection sampling, without imposing huge additional costs, since $g$ is univariate. One such method is laid out in Appendix~\ref{appA}.
\end{remark}

\section{Illustration} \label{appl}

\noindent
In this section, we consider a  real dataset to illustrate the improvements that can be  achieved by the Haar PX-DA algorithm  over the AC-DA algorithm. For this purpose,  we use the Lupus dataset from Table 1 in Van Dyk and Meng  \cite{vanDyk:Meng:2001}. This dataset consists of triplets $(y_i,x_{i1},x_{i2})$,  $i = 1,\dots,55$,  where $x_{i1}$ and $x_{i2}$ are covariates indicating levels of certain antibodies and $y_i$ is an indicator for the presence of latent membranous lupus nepthritis with 1 for presence and 0 for absence for the $i$th individual. (The dataset is also included in the \texttt{R} \cite{R}  package \texttt{TruncatedNormal} by Botev \cite{TruncatedNormal}.) Note that $\betab$ has 
$p=3$ components, including one intercept term. For our analysis, we considered a  $g$-prior with $g = 3.499999$ and prior mean 0, which ensures that the AC-DA Markov chain is trace-class. (See Remark \ref{gprior}.) Note that because the prior mean is assumed to be zero, $B(\zb)$ in (\ref{AzBz}) is also 0.  This makes the extra steps in the Haar PX-DA algorithm highly economical (see Remark \ref{B=0}).  The initial value for $\betab$ was taken to be $(-1.778,4.374,2.428)^T$, which is the maximum likelihood estimate. To facilitate comparison, along with the two $g$-prior based algorithms, we also consider the AC-DA and Haar PX-DA algorithms based on the improper flat prior from Roy and Hobert \cite{Roy:Hobert:2007}. All computations were done in \texttt{R} \cite{R} and the packages \texttt{truncnorm} \cite{truncnorm} and \texttt{TruncatedNormal} \cite{TruncatedNormal} were used.

The AC-DA algorithm is known to be extremely slow for the Lupus data (see e.g. 
Roy and Hobert \cite{Roy:Hobert:2007}; 
Pal, Khare and Hobert \cite{Pal:Khare:Hobert:2013}). Hence, all the four algorithms 
(AC-DA and Haar PX-DA, each with proper and improper priors)  were run with a 
burn-in period of $2 \times 10^6$ iterations. The next $10^6$ iterations were used to 
obtain the auto-correlations and running means for the two (non-intercept) regression 
coefficients $\beta_1$ and $\beta_2$ for all four chains. These autocorrelations and 
running means provide a natural way of evaluating 
convergence/efficiency of the DA and Haar PX-DA Markov chains.

We first compare the relative performances of all four chains together. 
Figure~\ref{acfplot1} and \ref{acfplot2} exhibit the auto-correlations  and 
Figure~\ref{runmeanplot1} and \ref{runmeanplot2}, the running  means,  for  $\beta_1$ 
and $\beta_2$ respectively. Observe the remarkably smaller autocorrelations for the 
chains based on the proper prior compared to those based on the improper flat prior 
shown in Figure~\ref{acfplot1} and \ref{acfplot2}. For instance, for $\beta_1$ in 
Figure~\ref{acfplot1}, note that all autocorrelations are less than 0.5 for the two proper 
prior chains while it takes 17 lags for the improper Haar PX-DA chain to achieve such 
an autocorrelation (and the improper AC-DA chain never reaches that value in the first 
fifty lags). The autocorrelation plots for $\beta_2$ show similar patterns in 
Figure~\ref{acfplot2}. In both Figure~\ref{acfplot1} and \ref{acfplot2}, the 
autocorrelations for the AC-DA proper prior chain almost coincide with those for the 
Haar PX-DA proper prior chain. Again, observe the noticeably better performances in 
terms of stability of running means for the chains based on the proper prior in 
Figure \ref{runmeanplot1} and \ref{runmeanplot2}. In the scales used in those two plots, 
the proper prior chains appear almost as coincidental horizontal straight lines. In 
contrast, on the same scales, the improper Haar PX-DA chain shows moderate, and the 
improper AC-DA chain, significant, fluctuations till 300,000 and 700,000 iterations 
respectively for both $\beta_1$ and $\beta_2$.\footnote{An interesting feature 
displayed in Figure~\ref{runmeanplot1} and \ref{runmeanplot2} is the significant 
differences between the estimated values of the parameters obtained from the proper 
prior chains and improper prior chains. These differences result from the fact that the 
proper prior and the improper prior lead to different posterior distributions.} These 
metrics indicate the noteworthy superiority (in terms of efficiency as well as 
convergence) of the chains based on the proper prior over those based on the 
improper flat prior in the current setting. 
\begin{figure}[!htpb]
	\centering
	\hspace*{\fill}%
	\subcaptionbox{Autocorrelation plots for $\beta_1$ values \label{acfplot1}}%
	[0.485 \linewidth]{\includegraphics[height=2.25in]{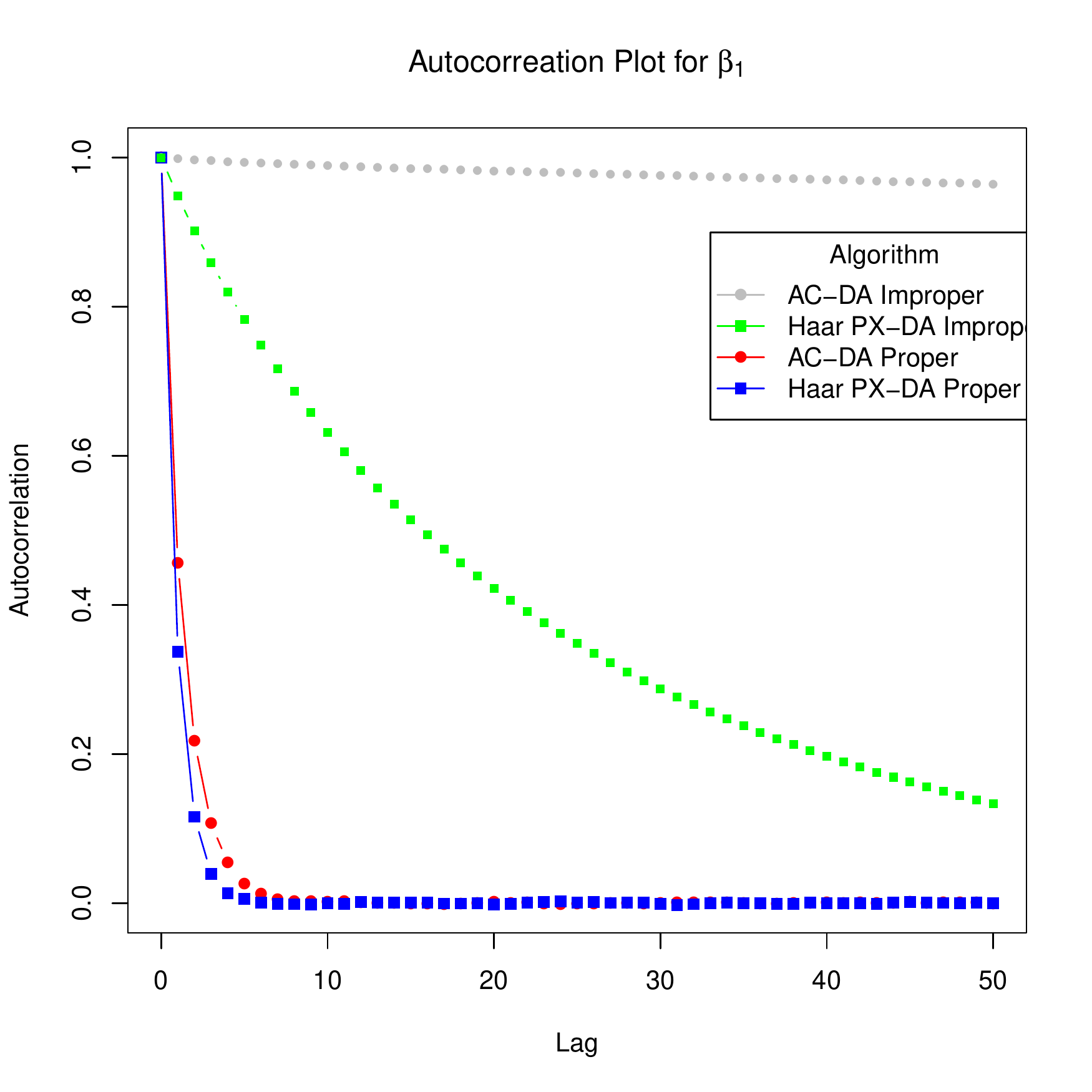}} 
	\hfill 
	\subcaptionbox{Running mean plots for $\beta_1$ values \label{runmeanplot1}}
	[.485\linewidth]{\includegraphics[height=2.25in]{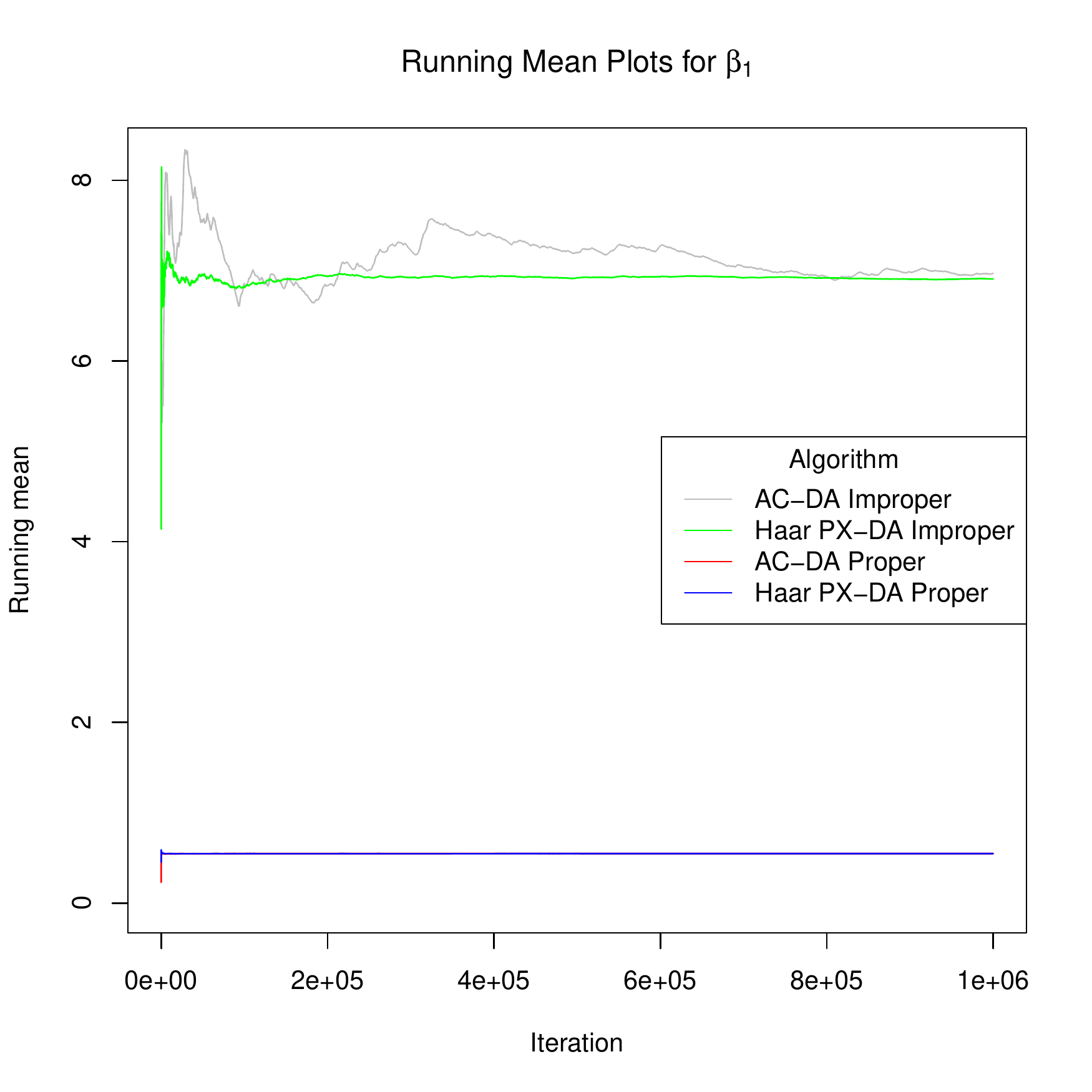}}
    \hspace*{\fill}%
	\caption{Convergence plots for the regression coefficient $\beta_1$ for AC-DA and Haar PX-DA chains, corresponding to the proper and improper priors, applied to the lupus data.}  
\end{figure}

\begin{figure}[!htpb]
	\centering
	\hspace*{\fill}%
	\subcaptionbox{Autocorrelation plots for $\beta_2$ values \label{acfplot2}}%
	[.485\linewidth]{\includegraphics[height=2.25in]{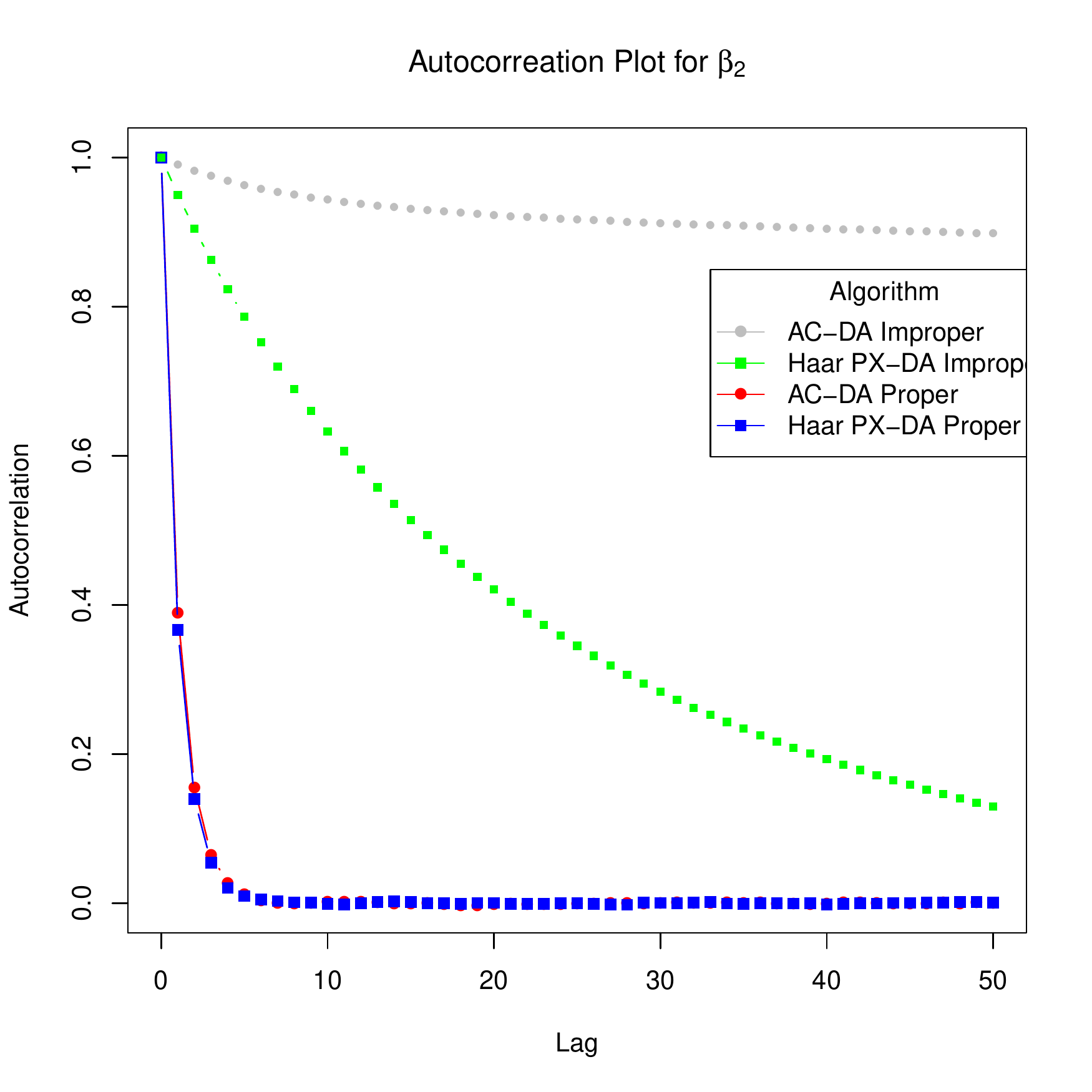}} 
	\hfill
	\subcaptionbox{Running mean plots for $\beta_2$ values \label{runmeanplot2}}
	[.485\linewidth]{\includegraphics[height=2.25in]{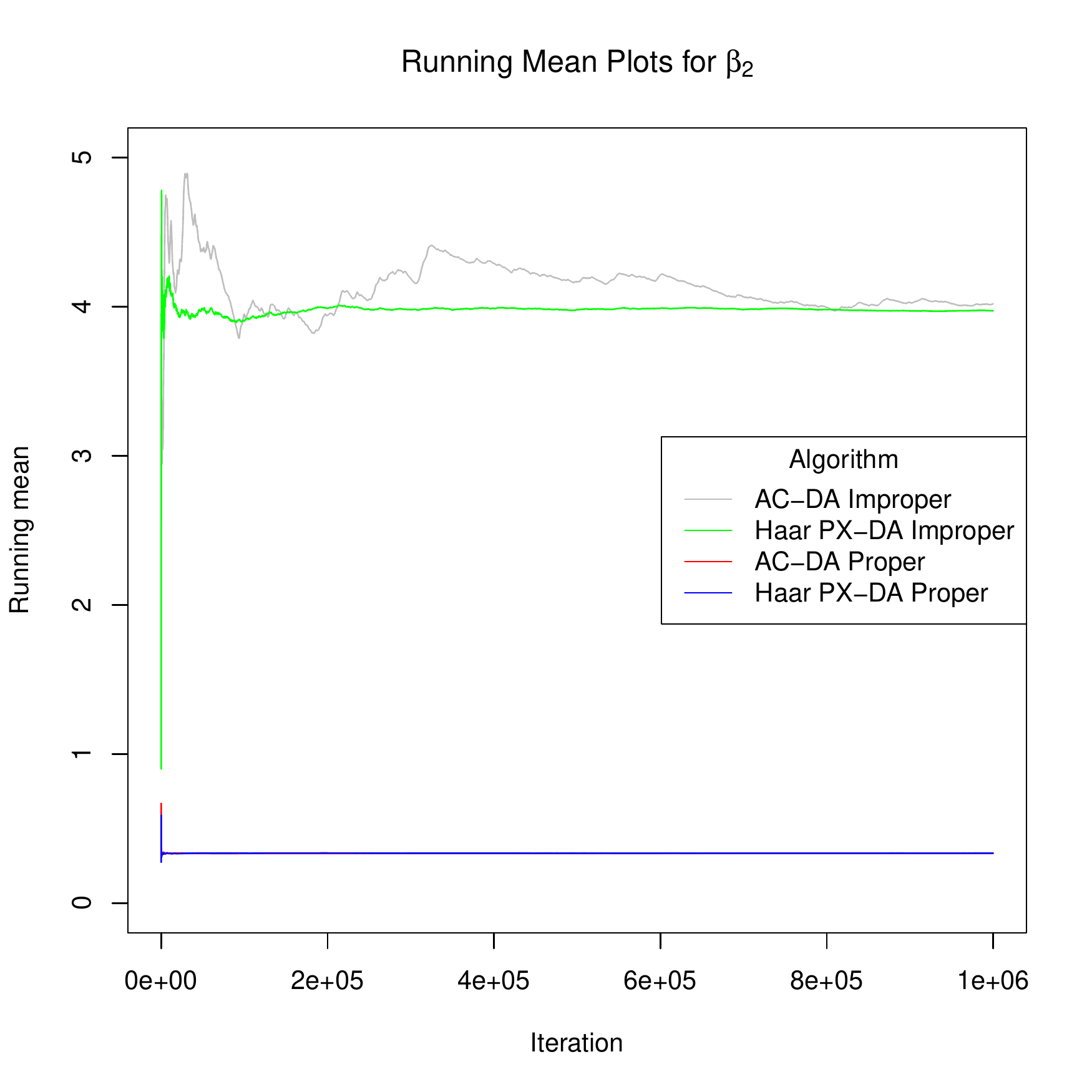}} 
	\hspace*{\fill}%
	\caption{Convergence plots for the regression coefficient $\beta_2$ for AC-DA and Haar PX-DA chains, corresponding to the proper and improper priors, applied to the lupus data.} 
\end{figure}

Because performances of the two proper prior chains are almost indistinguishable in 
the scales used in Figure~\ref{acfplot1},  \ref{runmeanplot1}, \ref{acfplot2} and 
\ref{runmeanplot2}, we take a closer look at these two chains to facilitate comparison. 
In particular, Figure~\ref{acfplot1.1} and \ref{acfplot2.1} display the autocorrelations 
and Figure~\ref{runmeanplot1.1} and \ref{runmeanplot2.1}, the running means, for 
$\beta_1$ and $\beta_2$ respectively, in appropriately chosen scales for the proper prior chains. Note that the 
autocorrelations are almost identical and the running means show very similar 
patterns in terms of stability for $\beta_2$ in Figure~\ref{acfplot2.1} and 
\ref{runmeanplot2.1} (even in the adjusted scale). On the other hand, 
Figure~\ref{acfplot1.1} and \ref{runmeanplot1.1} demonstrate a slightly more 
significant dominance of the Haar PX-DA chain over the AC-DA chain for $\beta_1$. 

\begin{figure}[!htpb]
	\centering
	\hspace*{\fill}%
	\subcaptionbox{Autocorrelation plots for $\beta_1$ values \label{acfplot1.1}}%
	[.485\linewidth]{\includegraphics[height=2.25in]{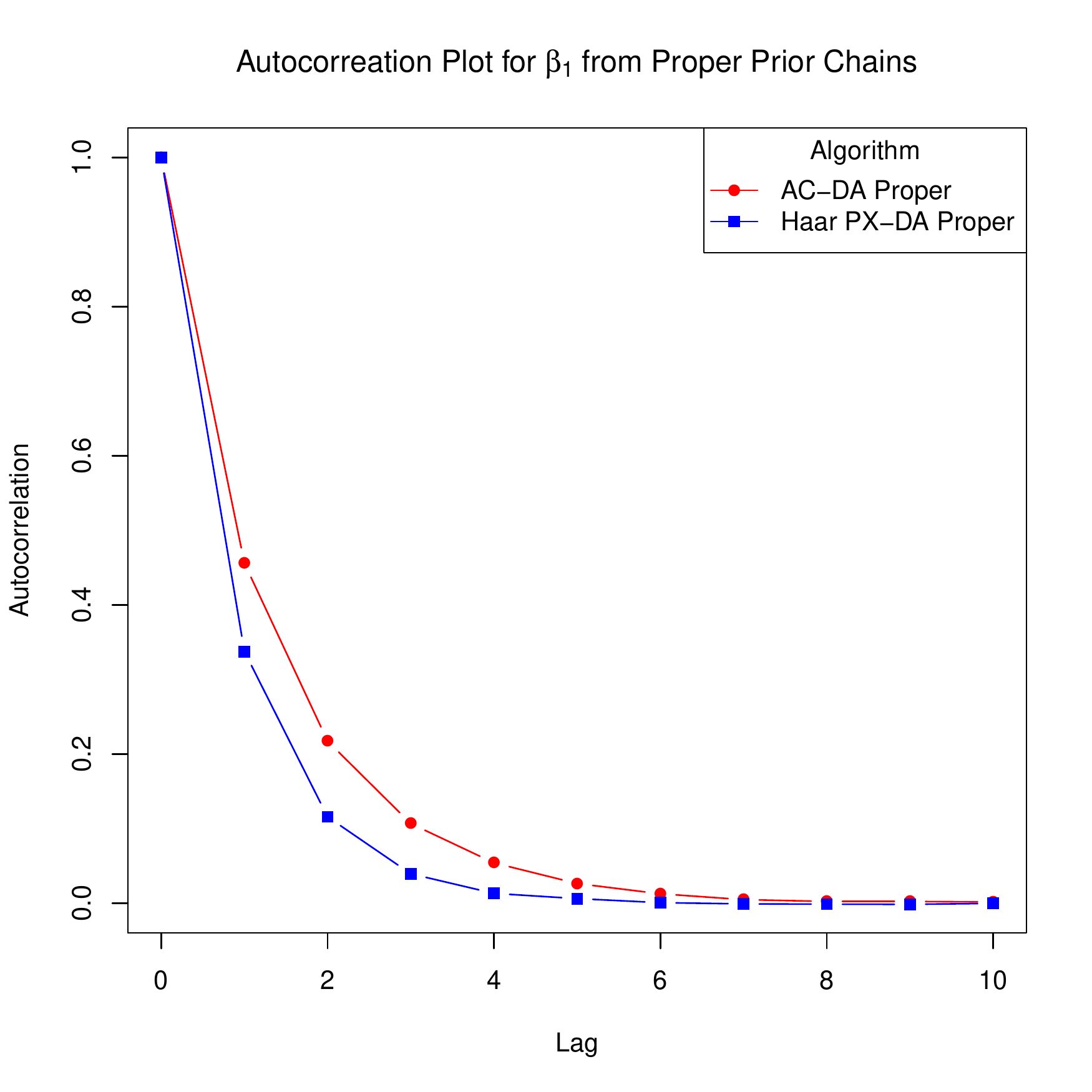}} 
	\hfill
	\subcaptionbox{Running mean plots  (on a finer scale) for $\beta_1$ values \label{runmeanplot1.1}}
	[.485\linewidth]{\includegraphics[height=2.25in]{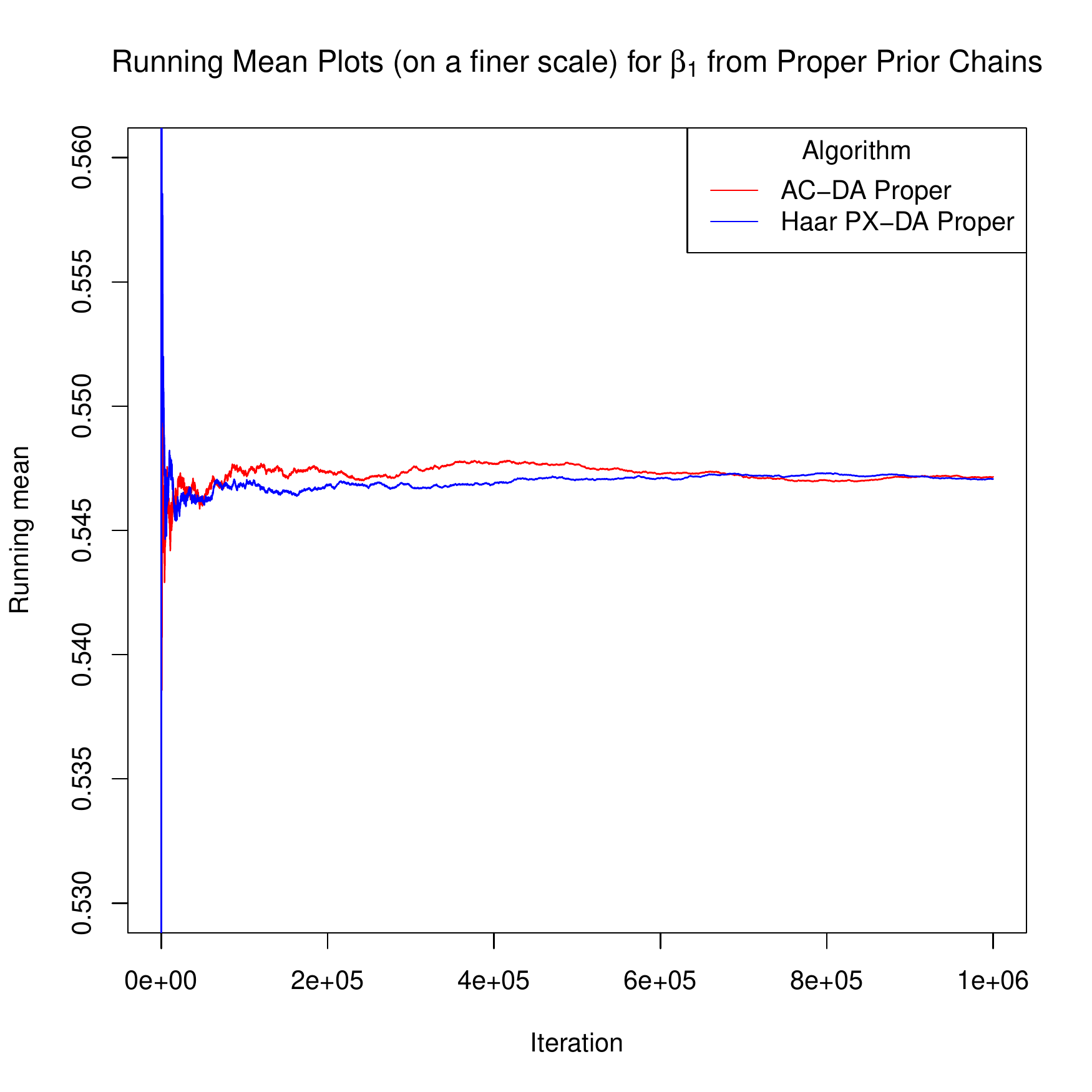}} 
	\hspace*{\fill}%
	\caption{Convergence plots (on a finer scale) for the regression coefficient $\beta_1$ for AC-DA and Haar PX-DA chains corresponding to the proper prior, applied to the lupus data.}  
\end{figure}

\begin{figure}[!htpb]
	\centering
	\hspace*{\fill}%
	\subcaptionbox{Autocorrelation plots for $\beta_2$ values \label{acfplot2.1}}%
	[.485\linewidth]{\includegraphics[height=2.25in]{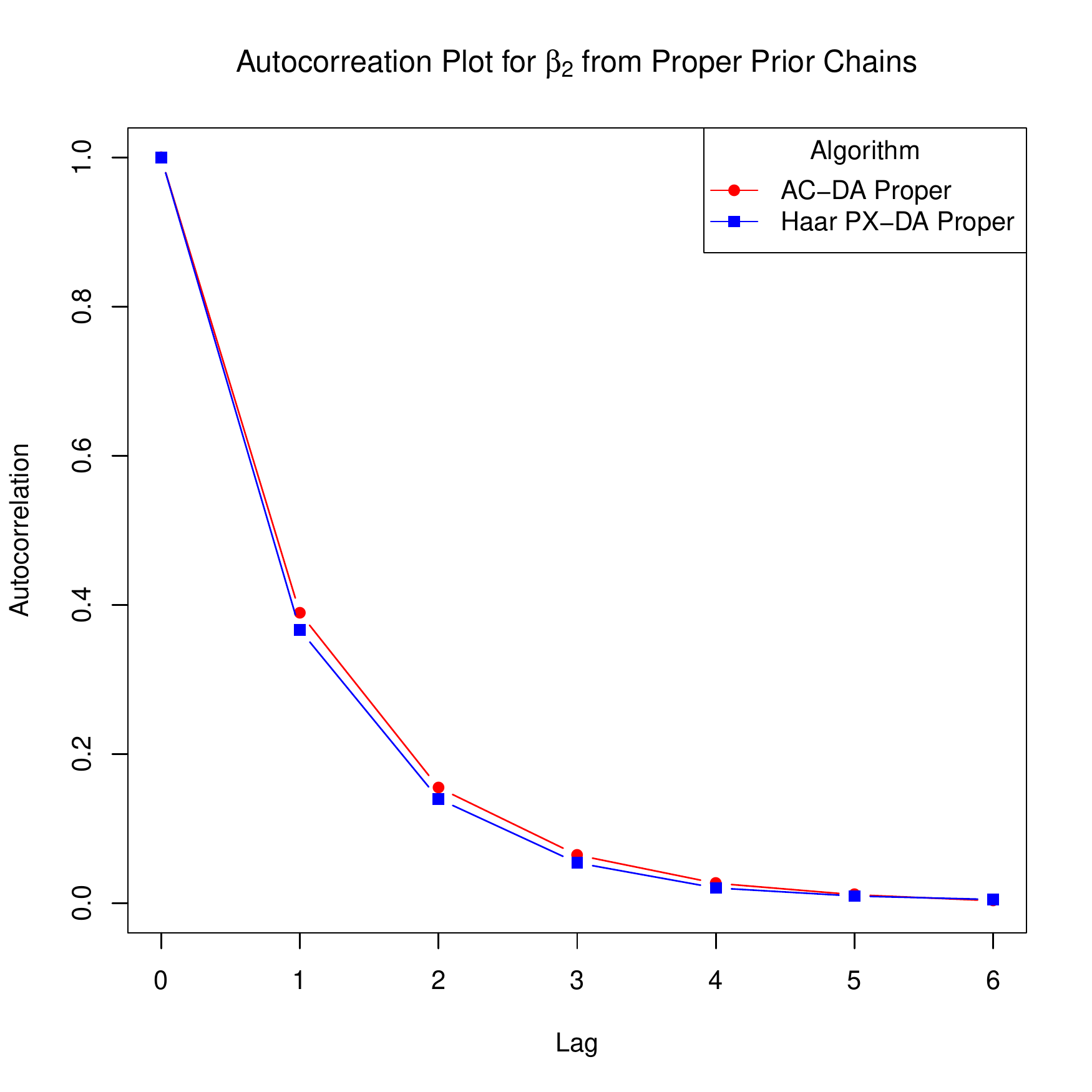}} 
	\hfill
	\subcaptionbox{Running mean plots (on a finer scale) for $\beta_2$ values \label{runmeanplot2.1}}
	[.485\linewidth]{\includegraphics[height=2.25in]{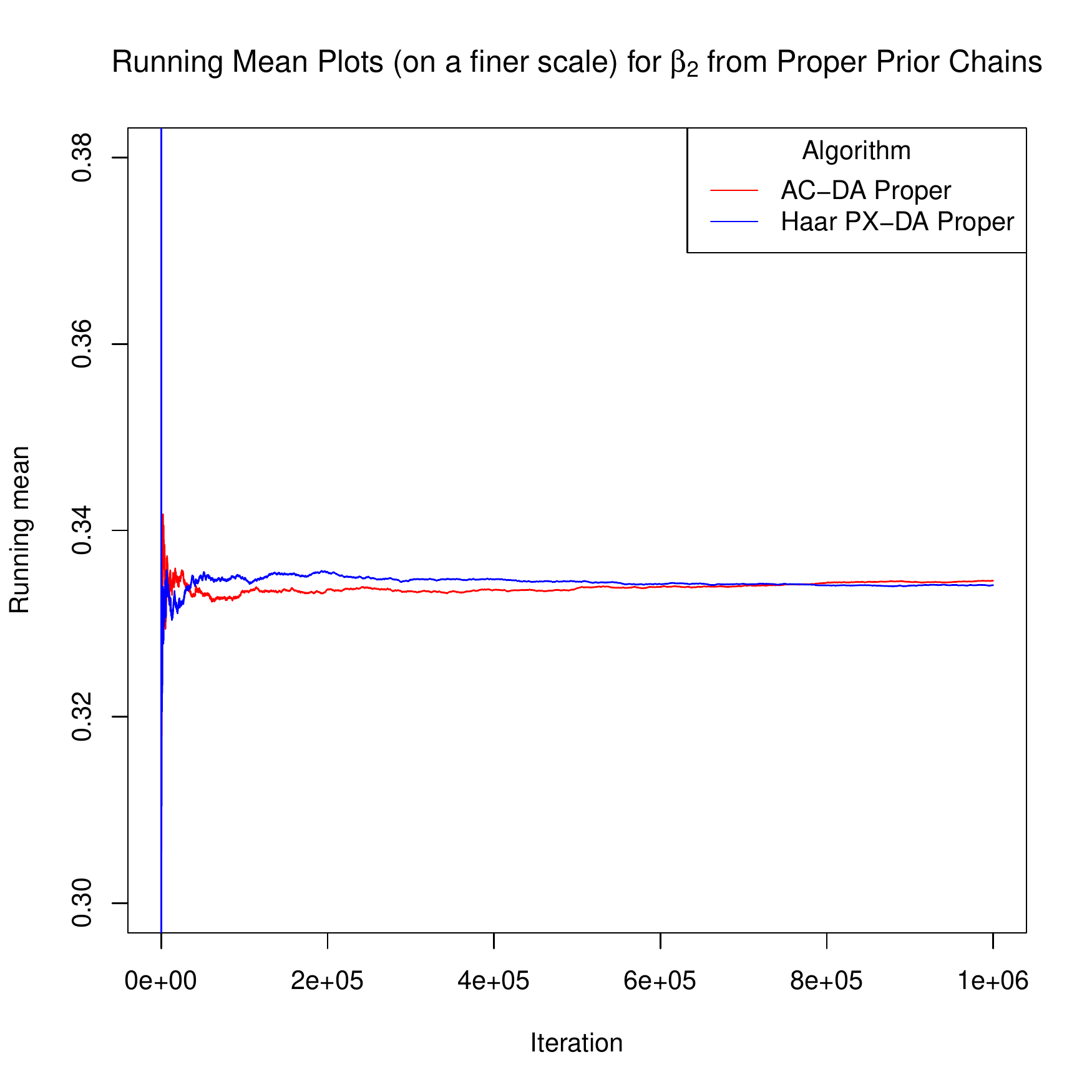}} 
	\hspace*{\fill}%
	\caption{Convergence plots (on a finer scale) for the regression coefficient $\beta_2$ for AC-DA and Haar PX-DA chains corresponding to the proper prior, applied to the lupus data.}  
\end{figure}

Thus, to summarize, it can be concluded that  in terms of convergence, the proper Haar PX-DA chain is the best among the four. Taking into account the practically insignificant amount of time needed to run the extra step, the Haar PX-DA algorithm with proper $g$-prior ($g = 3.499999$) is therefore undoubtedly the best choice among the all four algorithms considered in the current setting (the AC-DA algorithm based on the same prior being  a close competitor). 

\begin{appendix}

\section{Technical Results} \label{appB}
\begin{prop}
\label{prop1}
For any matrix $B \in \R^{n \times p}, B \neq 0_{n \times p}$, and any positive real number $\tau$, all eigenvalues of  $B\: ( B^TB + \tau I_p)^{-1} \: B^T$ lie within $[0,1)$, with at least one eigenvalue strictly positive.
\end{prop} 

\begin{proof}
We shall consider the cases $n \geq p$ and $n < p$ separately.
\subsection*{Case I : $n \geq p$}
Consider the following singular value decomposition:
\begin{align}
\underset{n \times p}{B} = \underset{ n \times p}{ U} \; \underset{ p \times p}{ D} \; {\underset{ p \times p}{ V}}^T
\end{align}
where $V \in \R^{p \times p}$ is orthogonal, $D \in \R^{p \times p}$ is diagonal, say $D = \diag(d_1, \cdots, d_p)$ where one or more (but not all) $d_i$'s may be equal to zero, $U \in \R^{n \times p}$ is a matrix with orthogonal columns and let 
\[
\underset{n \times n}{U^*} = \left( \underset{n \times p}{U} \left| \underset{n \times (n - p)}{\Ut} \right. \right) \
\]
be orthogonal in $\R^{n \times n}$. Then
\[
B^T B = V D^2 V^T \implies B^T B + \tau I_p = V D^2 V^T + \tau V V^T = V(D^2 + \tau I_p) V^T 
\]
So that, 
\[
( B^T B + \tau I_p )^{-1} = V(D^2 + \tau I_p)^{-1} V^T = V \left(\frac{I_p}{D^2 + \tau I_p}\right) V^T
\]
where, for diagonal matrices 
\[
M_{k \times k} = \diag(m_1, \cdots, m_k) \text{ and } N_{k \times k} =\diag(n_1, \cdots, n_k)
\] 
with $n_i \neq 0$ for all $i = 1,\cdots,k$, we define
\[
\frac{M}{N} := \diag\left( \frac{m_1}{n_1}, \cdots, \frac{m_k}{n_k} \right).
\]

\noindent Therefore,
\begin{align*}
B\: ( B^TB + \tau I_p)^{-1} \: B^T &= U D V^T V \left(\frac{I_p}{D^2 + \tau I_p}\right) V^T V D U^T \\
	&= U \left(\frac{D^2}{D^2 + \tau I_p}\right) U^T \\
	&= U \left(\frac{D^2}{D^2 + \tau I_p}\right) U^T + 0\: \Ut  \Ut^T
\end{align*}
which shows that the eigenvalues of $B\: ( B^TB + \tau I_p)^{-1} \: B^T$ are:
\[
\begin{cases}
0, & \text{with multiplicity} (n-p)\\
\frac{d_i^2}{\tau + d_i^2}, & i = 1, \cdots, p
\end{cases}.
\]

\subsection*{Case II: $n < p$}
For this case, consider following singular value decomposition:
\begin{align}
\underset{n \times p}{B^T} = \underset{ p \times n}{ U} \; \underset{ n \times n}{ D} \; {\underset{ n \times n}{ V}}^T
\end{align}
where as before (but now with different dimensions) $V \in \R^{n \times n}$ is orthogonal, $D \in \R^{n \times n}$ is diagonal, say $D = \diag(d_1, \cdots, d_n)$ where one or more (but not all) $d_i$'s may be equal to zero, and $U \in \R^{p \times n}$ is a matrix with orthogonal columns and
\[
\underset{p \times p}{U^*} = \left( \underset{p \times n}{U} \left| \underset{p \times (p - n)}{\Ut} \right. \right) 
\]
is orthogonal in $\R^{p \times p}$. Then 
\begin{align*}
 B^T B &= U D^2 U^T \\ 
\implies B^T B + \tau I_p &= U D^2 U^T + \tau U U^T + \tau \Ut \Ut^T \\ 
&= U(D^2 + \tau I_p) U^T + \tau \Ut \Ut^T 
\end{align*}
so that
\[
(B^T B + \tau I_p)^{-1} = U \left( \frac{I_p}{D^2 + \tau I_p} \right) U^T + \frac{1}{\tau} \Ut \Ut^T
\]
and therefore
\begin{align*}
B (B^T B + \tau I_p)^{-1} B^T &= V D U^T U \left( \frac{I_p}{D^2 + \tau I_p} \right) U^T U D V^T \\
& \qquad + \frac{1}{\tau}  V D U^T \Ut \Ut^T U^T U D V^T \\
&= V \left( \frac{D^2}{D^2 + \tau I_p} \right) V^T
\end{align*}
which means that the eigenvalues of $B\: ( B^TB + \tau I_p)^{-1} \: B^T$ are $d_i^2/(\tau + d_i^2)$, $i = 1, \cdots, n$.\\

\noindent Thus, in either case, all eigenvalues of $B\: ( B^TB + \tau I_p)^{-1} \: B^T$ lie within $[0,1)$ and at least one eigenvalue is positive as $B \neq 0_{n \times p}$. 
\end{proof}

\begin{prop} \label{prop2}
For any matrix $B \in \R^{n \times p}$ and any positive real number $\tau$,  $I_n - B\: ( B^TB + \tau I_p)^{-1} \: B^T$ is positive definite.
\end{prop} 

\begin{proof}
Note that the result is trivially true if $B = 0_{n \times p}$. So, without loss of generality we assume $B \neq 0_{n \times p}$. Let $\lambda_1,\cdots,\lambda_n$ denote the eigenvalues of $M := B ( B^TB + \tau I_p)^{-1} B^T$.
Then there exits an orthogonal matrix $U \in \R^{n \times n}$ such that $M = U \Lambda U^T \implies U^T M U = \Lambda$, where $\Lambda = \diag(\lambda_1,\cdots,\lambda_n)$. Note that
\[
U^T(I_n - M)U = U^TU - U^TMU = I_n - \Lambda = \diag\{1- \lambda_1,\cdots,1-\lambda_n\}
\]
which implies that the eigenvalues of $I_n - M$ are $1- \lambda_1,\cdots,1-\lambda_n$. This completes the proof since it follows from Proposition~\ref{prop1} that $\lambda_i \in [0,1) \implies 1-\lambda_i \in (0,1]$ for all $i = 1,\cdots,n$.
\end{proof}

\begin{remark}
Proposition~\ref{prop1} and Proposition~\ref{prop2} are essentially generalizations of 
the (first halves of) Lemma 4 and Lemma 5 in 
Roman and Hobert \citep{roman:hobert:2015}, where by exhibiting explicit forms for 
the eigenvalues of a matrix of the form $B\: ( \kappa B^TB + \Sigma^{-1})^{-1} \: 
B^T$, with $\kappa > 0$, $B \in \R^{n \times p}$, and $\Sigma \in \R^{p \times p}$ 
positive definite, the authors ultimately prove the positive definiteness of $I_n - 
\kappa B\: (\kappa B^TB + \Sigma^{-1})^{-1} \: B^T$. It is to be noted that these 
results in Roman and Hobert \citep{roman:hobert:2015} are derived under the 
assumption that $n \geq p$ and $\text{rank}(B) = p$, whereas 
Proposition~\ref{prop1} and Proposition~\ref{prop2} hold for \textit{any} $n$, $p$ 
and $B \in \R^{n \times p}$. 
\end{remark}

\begin{prop} \label{prop3}
Let $A, B \in \R^{m \times m}$ be symmetric matrices. If $A$ is positive definite, $B$ is positive semi definite  and $(k_n)_{n=1}^\infty$ is a sequence of positive numbers converging to zero, then for all large $n$, $A - k_n B$ is positive definite. 
\end{prop}

\begin{proof}
We first note that if $B = 0_{m \times m}$, then this result is trivially true for any $n \geq 1$. So, without loss of generality we shall assume $B \neq 0_{m \times m}$. Let $\xb \in \R^m \setminus \{\bm{0}\}$. Then, for any $n \geq 1$,
\begin{align*}
\frac{\xb^T(A - k_n B)\xb}{\xb^T\xb} = \frac{\xb^TA \xb}{\xb^T\xb} - k_n \frac{\xb^T B\xb}{\xb^T\xb}
\geq \lambdamin(A) - k_n \lambdamax(B)
\end{align*}
where $\lambdamin(A)$ and $\lambdamax(B)$ respectively denote the minimum and the maximum eigenvalues of $A$ and $B$. Since $A$ is positive definite and $B \neq 0$ is positive semi definite, therefore, both $\lambdamin(A)$ and $\lambdamax(B)$ are positive. Now, $k_n \downarrow 0$ means that there exists $N$ such that 
\[
n \geq N \implies k_n < \frac{1}{2} \: \frac{\lambdamin(A)}{\lambdamax(B)}
\]

\noindent Therefore, for all $n \geq N$,
\begin{align*}
\quad \frac{\xb^T(A - k_n B)\xb}{\xb^T\xb} &\geq \lambdamin(A) - k_n \lambdamax(B) \\ 
&> \lambdamin(A) - \frac12 \lambdamin(A) = \frac12 \lambdamin(A) > 0 \\
\implies \xb^T(A - k_n B)\xb &> 0
\end{align*}
This completes the proof since $\xb \in \R^m \setminus \{\bm{0}\}$ is arbitrary. 
\end{proof}

\section{Drawing observations from a density proportional to $w^*(g)$} \label{appA}

Here we describe a method for drawing observations from the density $\tilde{w}^*(g) = w^*(g)/\int_0^\infty w^*(s)\:ds$. First, note that
\begin{align*}
w^*(g) \: dg &= g^{n-1} \exp \left[ -\frac{1}{2} \left\lbrace  A(\zb) g^2 - 2  B(\zb) g \right\rbrace \right] dg  \\
&= \frac{1}{2} u^{n/2 - 1} \exp \left[ -\frac{1}{2} \left\lbrace A(\zb) u - 2  B(\zb)\sqrt{u} \right\rbrace  \right] du \\
&\propto u^{n/2 - 1} \exp \left[ -\frac{1}{2} \left\lbrace A(\zb) u - 2  B(\zb)\sqrt{u} \right\rbrace  \right] du =:l(u) \: du
\end{align*}
where $u = g^2$. So if $u$ is an observation from the density $\tilde{l}(\cdot) = l(\cdot) /\int_0^\infty l(s)\:ds$, then  the corresponding observation from $\tilde{w}^*(\cdot)$ will simply be $g = \sqrt{u}$. Note that when $B(\zb) = 0$, $\tilde{l}$ reduces to the Gamma$\left(\frac{n}{2},\frac{2}{A(\zb)}\right)$ density (see Remark \ref{B=0}), from which drawing observations is effortless. When $B(\zb) \neq 0$,  one can use  rejection sampling techniques, where $l$ needs to be majorized by a constant (depending on $\zb$) multiple of some standard density. Such a majorizing function can be easily found in this setting. For example, observe that for any $\epsilon \in (0,1)$, 
\begin{align*}
l(u) &= u^{n/2 - 1} \exp \left[ -\frac{A(\zb)}{2} u \right]   \exp\left[ 2 \times \frac{\sqrt{\epsilon A(\zb)  u}}{\sqrt{2}} \times \frac{B(\zb)}{ \sqrt{2 \epsilon A(\zb)}} \right] \\
&\leq u^{n/2 - 1} \exp \left[ -\frac{A(\zb)}{2} u \right] \exp \left[ \frac{ \epsilon A (\zb)}{2} u \right] \exp\left[\frac{B(\zb)^2}{2 \epsilon A(\zb)} \right] \\
&= \exp \left[\frac{B(\zb)^2}{2 \epsilon A(\zb)} \right] u^{n/2 - 1} \exp \left[ -\frac{(1-\epsilon)A(\zb)}{2} u \right]\\
&= M(\zb) \: f_{\text{Gamma}\left(\frac{n}{2},\frac{2}{(1-\epsilon) A(\zb)}\right)}(u)
\end{align*}
where
\begin{align} \label{Mz}
M(\zb) =  \exp \left[\frac{B(\zb)^2}{2 \epsilon A(\zb)} \right] \Gamma(n/2) \left(\frac{ 2}{ (1-\epsilon) A(\zb)}\right)^{n/2}
\end{align}
and $f_{\text{Gamma}\left(\alpha, \beta\right)}$ denotes the density function of the gamma distribution with location $\alpha$ and scale $\beta$, and the inequality in the second line follows from the fact that $2ab \leq a^2 + b^2$, with $a = {\sqrt{\epsilon A(\zb)  u/2}}$ and $b={B(\zb)}/{ \sqrt{2 \epsilon A(\zb)}}$.\\

\noindent
Thus, to summarize, one procedure involving rejection sampling to generate an 
observation from $\tilde{w}^*(g)$, is as follows.
\begin{enumerate}
\item  If $B(\zb) = 0$, generate $u \sim {\text{Gamma}\left(\frac{n}{2},\frac{2}{A(\zb)}\right)}$. If $B(\zb) \neq 0$, fix $\epsilon \in (0,1)$ and
\begin{enumerate}[label = (1\alph*)]
\item \label{rejsamp1} generate $u \sim {\text{Gamma}\left(\frac{n}{2},\frac{2}{(1-\epsilon) A(\zb)}\right)}$ and calculate 
\[
\rho(u) = \frac{l(u)}{M(\zb) \: f_{\text{Gamma}\left(\frac{n}{2},\frac{2}{(1-\epsilon) A(\zb)}\right)}(u)}.
\] 
where $M(\zb)$ is as given in (\ref{Mz}).
\item \label{rejsamp2} perform a Bernoulli experiment with probability of success $\rho(u)$. If a success is achieved, retain $u$. Else return to \ref{rejsamp1}.
\end{enumerate}
\item Compute $g = \sqrt{u}$.
\end{enumerate}

\end{appendix}

\bibliographystyle{apa}
\bibliography{Reference}

\end{document}